\documentclass[11pt]{amsart}
\usepackage{graphicx} 
\usepackage{amssymb,amscd,amsmath,color,latexsym,epsfig,hyperref,mathtools, graphicx,enumitem}
\usepackage[all]{xy}
\usepackage[mathscr]{euscript}
\usepackage{tikz-cd}
\usepackage{tikz}
\usepackage{circuitikz}
\usepackage{float}

\usetikzlibrary{calc,trees,positioning,arrows,chains,shapes.geometric,%
    decorations.pathreplacing,decorations.pathmorphing,shapes,%
    matrix,shapes.symbols}

\tikzset{
>=stealth',
  punktchain/.style={
    rectangle,
    rounded corners,
    draw=black, thick,
    minimum height=3em,
    text centered,
    on chain},
  line/.style={draw, thick, <-},
  element/.style={
    tape,
    top color=white,
    bottom color=blue!50!black!60!,
    minimum width=8em,
    draw=blue!40!black!90, very thick,
    text width=10em,
    minimum height=3.5em,
    text centered,
    on chain},
  every join/.style={->, thick,shorten >=1pt},
  decoration={brace},
  tuborg/.style={decorate},
  tubnode/.style={midway, right=2pt},
}

\DeclareFontFamily{OT1}{rsfs}{}
\DeclareFontShape{OT1}{rsfs}{n}{it}{<-> rsfs10}{}
\DeclareMathAlphabet{\curly}{OT1}{rsfs}{n}{it}

\DeclareFontFamily{U}{mathb}{\hyphenchar\font45}
\DeclareFontShape{U}{mathb}{m}{n}{
      <5> <6> <7> <8> <9> <10> gen * mathb
      <10.95> mathb10 <12> <14.4> <17.28> <20.74> <24.88> mathb12
      }{}
\DeclareSymbolFont{mathb}{U}{mathb}{m}{n}

\makeatletter
\newcommand{\eqnum}{\refstepcounter{equation}\textup{\tagform@{\theequation}}}
\makeatother

\newcommand\C{\mathbb C}
\newcommand\Q{\mathbb Q}
\newcommand\R{\mathbb R}

\newcommand\Z{\mathbb Z}
\renewcommand\AA{\mathbb A}

\newcommand\PP{\mathbb P}

\newcommand\TT{\mathbb T}

\renewcommand\P{\mathbf P}

\newcommand\cO{\mathcal O}

\newcommand\cC{\mathcal C}

\newcommand\cE{\mathcal E}

\newcommand\cM{\mathcal M}

\newcommand\cT{\mathcal T}

\newcommand\cZ{\mathcal Z}

\makeatletter
\newcommand{\so}{\ \ext@arrow 0359\Rightarrowfill@{}{\hspace{3mm}}\ }
\makeatother

\newcommand\into{\hookrightarrow}
\newcommand\INTO{\ \ar@{^(->}[r]<-.2ex>}

\newfont{\bigtimesfont}{cmsy10 scaled \magstep5}
\newcommand{\bigtimes}{\mathop{\lower0.9ex\hbox{\bigtimesfont\symbol2}}}
\renewcommand\={\ =\ }

\DeclareMathSymbol{\lefttorightarrow}{3}{mathb}{"FC}
\DeclareMathSymbol{\righttoleftarrow}{3}{mathb}{"FD}

\DeclareMathOperator*{\bigboxtimes}{\vcenter{\hbox{\scalebox{2}{$\boxtimes$}}}}

\newcommand\pt{\operatorname{pt}}

\newcommand\ev{\operatorname{ev}}

\newcommand\rk{\operatorname{rank}}
\newcommand\vir{\operatorname{vir}}
\newcommand\vd{\operatorname{vdim}}
\newcommand\vcd{\operatorname{vcodim}}

\newcommand\Aut{\operatorname{Aut}}

\newcommand\Cone{\operatorname{Cone}}
\newcommand\Bl{\operatorname{Bl}}

\newcommand\Eq{\operatorname{Eq}}
\newcommand\rel{\operatorname{rel}}
\newcommand\fix{\operatorname{fix}}
\newcommand\res{\operatorname{res}}
\def\wtil{\widetilde}

\newcommand\Sing{\operatorname{Sing}}
\newcommand\st{\operatorname{st}}
\newcommand\edg{\operatorname{edge}}

\newcommand\val{\mathrm{val}}
\newcommand\un{\mathrm{un}}
\newcommand\gl{\mathrm{gl}}
\newcommand\Gat{\mathrm{Gat}}

\def\Bl{\mathrm{Bl}}

\def\pt{\mathrm{pt}}
\def\ev{\mathrm{ev}}

\newcommand{\set}[1]{\{ #1 \}}
\newcommand{\what}[1]{\widehat{#1}}
\def\Sch{\mathrm{Sch}}

\def\rH{\mathrm{H}}

\def\fM{\mathfrak{M}}

\def\fW{\mathfrak{W}}

\def\bB{\mathbf{B}}

\def\bD{\mathbf{D}}

\def\Ver{\mathrm{Ver}}
\def\Edg{\mathrm{Edg}}
\def\Leg{\mathrm{Leg}}

\newcommand\beq[1]{\begin{equation}\label{#1}}
\newcommand\eeq{\end{equation}}
\newcommand\beqa{\begin{eqnarray*}}
\newcommand\eeqa{\end{eqnarray*}}

\DeclareRobustCommand{\SkipTocEntry}[3]{}
\makeatletter
\newcommand\@dotsep{4.5}
\def\@tocline#1#2#3#4#5#6#7{\relax
  \ifnum #1>\c@tocdepth 
  \else
    \par \addpenalty\@secpenalty\addvspace{#2}%
    \begingroup \hyphenpenalty\@M
    \@ifempty{#4}{%
      \@tempdima\csname r@tocindent\number#1\endcsname\relax
    }{%
      \@tempdima#4\relax
    }%
    \parindent\z@ \leftskip#3\relax \advance\leftskip\@tempdima\relax
    \rightskip\@pnumwidth plus1em \parfillskip-\@pnumwidth
    #5\leavevmode #6\relax
    \leaders\hbox{$\m@th
      \mkern \@dotsep mu\hbox{.}\mkern \@dotsep mu$}\hfill
    \hbox to\@pnumwidth{\@tocpagenum{#7}}\par
    \nobreak
    \endgroup
  \fi}
\makeatother

\newtheorem{prop}{Proposition}[section]
\newtheorem{thm}[prop]{Theorem}
\newtheorem*{thm*}{Theorem}
\newtheorem{lem}[prop]{Lemma}
\newtheorem{coro}[prop]{Corollary}

\newtheorem{defn}[prop]{Definition}
\newtheorem{conj}[prop]{Conjecture}

\theoremstyle{definition}
\newtheorem{rmk}[prop]{Remark}

\newtheorem{exam}[prop]{Example}

\title[GW invariants of blow-ups]{Blow-up formulas of Gromov--Witten invariants and virtual cycles under positivity conditions}

\author{Sanghyeon Lee}
\address{Ajou University,
206 World cup-ro, Suwon, Republic of Korea}
\email{sanghyeon25@ajou.ac.kr}

\author{Seungjae Yun}
\address{Seoul National University, 1 Gwanak-ro, Seoul, Republic of Korea}
\email{for63434@snu.ac.kr}

\thanks{}

\date{}

\begin{document}

\begin{abstract} Let $\wtil X$ be the blow-up of a smooth projective variety $X$ along a smooth subvariety $Z \subset X$ satisfying a numerical positivity condition. We study virtual cycles of moduli spaces of stable maps to $\wtil X$. 

We lift the numerical blow-up formula of Chen and Du \cite{CD23} to a pushforward identity for virtual cycles of moduli spaces of genus-zero stable maps to $\wtil{X}$ and $X$. Using virtual dimension calculations for the fixed loci and a free-leaf contraction morphism, we show that all correction terms in the master space localization formula vanish. Variants of this method yields further blow-up formulas, including a generalization of Gathmann’s formula \cite{Gat01}. 

We also give a counterexample to the point-blow-up conjecture in higher genus \cite{Hu00}, and establish an absolute--relative correspondence in all genera under positivity and dimension assumptions, comparing the Gromov--Witten theories of $(\wtil{X},E)$ and $X$, where $E$ is the exceptional divisor.

\end{abstract}

\maketitle

\setcounter{tocdepth}{1}
\tableofcontents

\section{Introduction}

The relationship between Gromov--Witten (GW) invariants of a smooth projective variety $X$ and those of its blow-up $\wtil{X}$ along a smooth subvariety $Z \subset X$ has been extensively studied in both algebraic geometry and symplectic geometry.

Given primary insertions $\gamma_1,\cdots,\gamma_k \in \rH^*(X;\Z)$, a genus $g$, and a degree $\beta \in \rH_2(X;\Z)$, the Gromov--Witten invariant $$\langle \gamma_1,\cdots, \gamma_k \rangle^X_{g,\beta}$$ heuristically counts nodal curves of genus $g$ and degree $\beta$ meeting representatives of the Poincar\'e duals of the insertions classes at the marked points.
One may consider the ideal situation where we can choose suitable $A_1,\cdots,A_k$ so that such curves $C$ do not intersect the blow-up locus $Z$. Let $\pi : \wtil{X} \to X$ be the projection morphism. Then, $\pi^{-1}(C)$ intersects $\pi^{-1}(A_1),\cdots, \pi^{-1}(A_k)$. Note that the curve $\pi^{-1}(C)$ has the degree $\pi^! \beta \in \rH_2(\wtil{X};\Z)$.

Conversely, consider a genus $g$ nodal curve $\wtil{C} \subset \wtil{X}$ with degree $\pi^! \beta$ that intersects with $\pi^{-1}(A_1),\cdots, \pi^{-1}(A_k)$. Let $E$ be the exceptional divisor of $\wtil{X}$. Since $\pi^! \beta \cdot E = 0$, if  $\wtil{C}$ is sufficiently general, it does not intersect $E$. Then $\pi(\wtil{C})$ becomes a curve of degree $\beta$ that intersects $A_1,\cdots,A_k$. The GW invariant $\langle \pi^*\gamma_1, \cdots, \pi^*\gamma_k \rangle^{\wtil{X}}_{g, \pi^! \beta}$ counts such $\wtil{C}$.

Thus, one expects the two GW invariants to be identical under suitable assumptions.
\begin{equation}\label{conj:blowup1}
\langle \gamma_1,\cdots, \gamma_k \rangle^X_{g,\beta} = \langle \pi^*\gamma_1, \cdots, \pi^*\gamma_k \rangle^{\wtil{X}}_{g, \pi^! \beta}
\end{equation}
There are various results in this direction, on both the algebraic geometry and symplectic geometry sides.

\begin{rmk}
Throughout this paper, we assume that the blow-up locus $Z$ is connected, to make the statements simple. But all of our results straightforwardly generalize to the case when $Z$ is not connected.
\end{rmk}

\subsection{Previous results on the blow-up formula}
We first review results on this conjecture obtained by algebro-geometric methods. In \cite[Theorem 2.1, Lemma 2.2]{Gat01}, Gathmann proved that \eqref{conj:blowup1} holds when $X$ is a convex variety, $Z$ is a finite set of points, and the genus is $0$. Furthermore, Gathmann proved the pushforward formula for the virtual cycles of the stable map spaces, $\pi_*[\overline{\cM}_{0,k}(\wtil{X},\pi^!\beta)]^{\vir} = [\overline{\cM}_{0,k}(X,\beta)]^{\vir}$ under the same conditions. (This pushforward formula is stronger than the numerical formula \eqref{conj:blowup1}.)

In \cite[Theorem 1.6]{Lai09}, assuming that the blow-up locus is a transverse intersection of two submanifolds of a compact homogeneous space and the genus is $0$, Lai proved that the same pushforward formula for the virtual cycles holds. Lai used virtual pushforwards \cite{Man12b} and the degeneration formula \cite{Li02}.

In \cite[Theorem 1.3]{CD23}, Chen and Du proved the numerical formula \eqref{conj:blowup1} in genus zero and the normal bundle $N_{Z/X}$ satisfies the following positivity condition \cite[Definition 1.1]{CD23}:
\begin{equation}\label{eq:positivity1}
c_1(N_{Z/X}) \cdot f_*[\PP^1] + \mathrm{codim}(Z,X) > 0 \quad \textrm{for any regular map $f : \PP^1 \to Z$}.
\end{equation}
They used virtual localization of a master space.

On the other hand, there are also results on the conjecture obtained by symplectic-geometric methods. In \cite[Theorem 1.2 and Theorem 1.3]{Hu00} and \cite[Theorem, p. 345]{Hu01}, Hu proved the conjecture in the following cases, assuming $X$ is a compact symplectic manifold.
\begin{itemize}
    \item [(I)] $g\leq1$ and $Z$ is a point.
    \item [(II)] $\dim_{\R}(X) \le 6$, $Z$ is a point, and $k\ge 1$ or $\beta \neq 0$.
    \item [(III)] $g=0$, $Z$ is a point or a smooth curve $C$ (a $2$-dimensional symplectic submanifold) satisfying $c_1(T_X) \cap [C] \ge 0$.
    \item [(IV)] $g=0$, $X$ is semipositive, $Z$ is a smooth surface (a $4$-dimensional symplectic submanifold), and $$\gamma_1|_Z = \cdots = \gamma_k|_Z = 0.$$   
\end{itemize}

\begin{thm}[Hu] \label{thm:Hublowup}
Equation \eqref{conj:blowup1} holds for the cases (I) to (IV) above.
\end{thm}
In Appendix~\ref{sec:Appblowup1}, we give an analogous algebro-geometric proof of Theorem \ref{thm:Hublowup} in a slightly more general setting.

\begin{rmk}
 \label{rmk: beta not zero}
 The assumption $k\ge 1$ or $\beta\neq 0$ in case (II) is missing in Hu's original work. As explicitly noted in later works \cite{HHKQ18}, this assumption is necessary. For example, the following simple computation gives a counterexample to case (II) without this assumption. Let $X$ be a smooth projective threefold, and let $\wtil X$ be its blow-up at a point. Recall that for $g\geq 2$,
 \begin{equation*}
     \langle\text{ }\rangle_{g,0}^{X}=\frac{(-1)^g}{2}\left(\int_X\left(c_3(X)-c_1(X)c_2(X)\right)\right)\int_{\overline\cM_{g}}\lambda_{g-1}^3
 \end{equation*}
 as stated in \cite[Section 2.1]{MNOP06}. With $\int_{\wtil X}c_3(\wtil X)=\int_Xc_3(X)+2$, we get
 \begin{equation*}
     \langle\text{ }\rangle_{g,0}^{\wtil X}=\langle\text{ }\rangle_{g,0}^{X}+(-1)^g\int_{\overline\cM_{g}}\lambda_{g-1}^{3}.
 \end{equation*}
  A similar computation appears in \cite{ILLW12}.
 \end{rmk}

Hu further conjectured in \cite[p. 711]{Hu00} that \eqref{conj:blowup1} will also hold when $g > 1$ and $Z$ is a point.  
\begin{conj}[Hu]
\label{conj:Hu}
    The equation \eqref{conj:blowup1} holds for every genus when $Z=\pt$ and $\beta\neq 0$ (the condition $\beta\neq 0$ is missing in Hu's original work; see Remark~\ref{rmk: beta not zero}).
\end{conj}

Meanwhile, for a convex variety $X$ and a finite set of points $Z$, Gathmann proved a different genus-zero blow-up formula \cite[Corollary 3.2]{Gat01}. It suffices to take $Z$ to be a single point. 
\begin{thm}[Gathmann] \label{thm:Gathmann1}
For $\beta\neq 0$,
\begin{equation} \label{eq:Gathmann1}
\langle \gamma_1,\cdots, \gamma_k, [\pt] \rangle^X_{0,\beta} = \langle \pi^*\gamma_1, \cdots, \pi^*\gamma_k \rangle^{\wtil{X}}_{0, \pi^! \beta - \ell}    
\end{equation}
where $\ell$ the homology class of a line in a fiber of the projective bundle $\PP(N_{Z/X}) \to Z$. 
\end{thm}
Later, Hu extended this result to any compact symplectic manifold $X$ \cite[Theorem 1.4]{Hu00}.

We can understand the geometric meaning of this theorem as follows. The GW invariant on the left-hand side of \eqref{eq:Gathmann1} heuristically counts curves $C$ which intersect $A_1,\cdots,A_k$ and the point $Z$. Now, consider the proper transform $\wtil{C} \subset \wtil{X}$ of this curve $C$. Then, $\wtil{C}$ is a degree $\pi^!\beta - \ell$ curve, intersecting $\pi^{-1}(A_1), \cdots, \pi^{-1}(A_k)$ and the exceptional divisor $E \subset \wtil{X}$. The GW invariant $$\langle \pi^*\gamma_1, \cdots, \pi^*\gamma_k, [E] \rangle^{\wtil{X}}_{0, \pi^! \beta - \ell}$$ counts such curves $\wtil{C}$. This coincides with the GW invariant on the right-hand side of \eqref{eq:Gathmann1} by the divisor equation. This
suggests the following formula holds for a general blow-up locus $Z$ (not necessarily a finite set of points) and $\beta\neq 0$:

\begin{equation} \label{conj:genGathmann}
\langle \gamma_1,\cdots, \gamma_k, [Z] \rangle^X_{0,\beta} = \langle \pi^*\gamma_1, \cdots, \pi^*\gamma_k \rangle^{\wtil{X}}_{0, \pi^! \beta - \ell}    
\end{equation}

\medskip

We can continue the geometric picture above as follows. Let us consider a curve $C$ that intersects $A_1,\cdots,A_k$ and meets $Z$ with total intersection multiplicity $M$. Then, the proper transform $\wtil{C}$ has degree $\pi^! \beta - M \ell$, and one may attempt to establish a correspondence as above. However, for a curve $\wtil{C} \subset \wtil X$ of degree $\pi^! \beta + M \ell$ where $M > 0$, there is no natural choice of a corresponding curve $C \subset X$ of degree $\beta$ whose proper transform is $\wtil{C}$. This suggests the vanishing when $M > 0$:
\begin{equation} \label{conj:vanishing}
\langle \pi^*\gamma_1, \cdots, \pi^*\gamma_k \rangle^{\wtil{X}}_{0, \pi^! \beta + M \ell} = 0.    
\end{equation}
Note that Gathmann proved that this vanishing when $X$ is convex and $Z$ is a finite set of points \cite[Proposition 3.1]{Gat01}. Recently, this was proven in \cite[Theorem 1.1]{CDF26} when $Z$ satisfies the positivity condition \eqref{eq:positivity1}, using orbifold Gromov--Witten theory and root stacks.

\subsection{Summary of results}
\label{sec:Introduction}
Let $X$ be a smooth projective variety of dimension $d$ and let $Z \subset X$ be a smooth subvariety of codimension $r \geq 2$.
We will first consider a blow-up formula for genus 0 Gromov--Witten invariants, under the positivity condition on the blow-up center $Z \subset X$. Recall the morphism $\rH_2(X;\Z) \to \rH_2(\wtil X;\Z)$ given by 
$$\beta \mapsto \pi^! \beta, \ \ \textrm{for $\beta \in \rH_2(X;\Z)$}.$$

In this paper, we prove the following blow-up formulas assuming positivity of the blow-up center $Z \subset X$. We denote the number of markings by $k$. We always assume that $\overline\cM_{0,k}(X,\beta)$ is nonempty.

\begin{enumerate}
\item We have
\begin{equation*}
\pi_*[\overline\cM_{0,k}(\wtil X,\pi^!\beta)]^{\vir}=[\overline\cM_{0,k}(X,\beta)]^{\vir}.   
\end{equation*}

\item
For $M > 0$, we have
\begin{equation*}
\pi_*[\overline\cM_{0,k}(\wtil X,\pi^!\beta+M\ell)]^{\vir}=0.
\end{equation*}

\end{enumerate}
These blow-up formulas are established in Theorems~\ref{thm:main1} and \ref{thm:main2}. Counterexamples to these theorems when the positivity assumption is omitted are given in Examples~\ref{exam:counterex for non-positive} and \ref{ex:counterex-vanishing}, respectively.

\begin{rmk}
In \cite{Lai09}, the author assumed that the normal bundle $N_{Z/X}$ is convex, which means that for every $f : \PP^1 \to Z$, $f^* N_{Z/X}$ splits as a direct sum of line bundles of non-negative degree over $\PP^1$. Therefore, convexity of $N_{Z/X}$ implies the positivity condition.
\end{rmk}
\smallskip
\noindent\textit{Idea of the proof.} We use a standard master space containing the moduli spaces of stable maps to $X$ and to $\wtil{X}$. Virtual localization produces contributions from its fixed loci (Section~\ref{sec: localization}). Among them, loci of sufficiently large virtual codimension do not contribute; this is where the positivity of the blow-up center $Z$ enters the argument. In fact, a standard computation shows that fixed loci with only stable vertex contributions do not contribute under this assumption. However, unstable vertices, especially univalent ones, prevent dimensional vanishing in general. To handle these loci, we factor the relevant morphism through a contraction of free leaves. The contraction is virtually smooth and has sufficiently large relative virtual dimension, yielding the required vanishing by virtual pushforward (Section~\ref{sec: Type IV vanishing}). \\

Variants of the master space construction also yield the relative blow-up formula in Theorem~\ref{thm: relative blow-up formula}. This generalizes the relative blow-up formula in \cite[Section 5.1]{Lai09}. As another application, we prove a generalization \eqref{conj:genGathmann} of Gathmann's formula under a positivity condition on the normal bundle $N_{Z/X}$, in Theorem~\ref{thm:Gathmann}. \\

As recalled above in Theorem \ref{thm:Hublowup}, equation \eqref{conj:blowup1} was proved in \cite{Hu00} and \cite{Hu01} in cases (I) to (IV) using degeneration. More precisely, Hu used relative GW invariants of $\overline\cM_{\Psi^\circ}(\wtil X, E)$ as an intermediate comparison. Here $\Psi^\circ$ is the admissible graph with a single vertex of genus $g$, $k$ legs, no roots, and degree $\pi^!\beta$. The following diagram is a schematic summary:
 \begin{equation}
 \label{eq:degn scheme}    \langle\gamma_1,\cdots,\gamma_k\rangle_{g,\beta}^{X}\leftrightarrow\langle\pi^*\gamma_1,\cdots,\pi^*\gamma_k\rangle_{\Psi^\circ}^{(\wtil X,E)}\leftrightarrow\langle\pi^*\gamma_1,\cdots,\pi^*\gamma_k\rangle_{g,\pi^!\beta}^{\wtil X}
 \end{equation}
 for $\gamma_i\in \rH^*(X;\Z)$.
 The first correspondence uses a degeneration formula connecting $X$ and $\wtil X\cup_E \what E$, where $\what E=\PP_Z(N_{Z/X}\oplus 1)$. The second correspondence uses a degeneration connecting $\wtil X$ and $\wtil X\cup_{E}\PP_{E}(\cO(-1)\oplus 1)$. \\

 We call the comparison arising from the first degeneration the absolute--relative correspondence. Theorem~\ref{thm: main Abs-Rel corr} extends this comparison from point centers to positive centers in genus at most one and lifts it to an identity of virtual cycles, with correction terms in the supportive genus-one case. Theorem~\ref{thm: all genus Abs-Rel corr} gives an analogous higher-genus identity under the assumptions  $\dim Z\geq4$ and $g$-positivity (see Definition~\ref{defn:positivity for all g}).\\
 
 For the second degeneration, the normal bundle $\cO_E(-1)$ does not satisfy the positivity assumptions used in our argument. We expect that controlling its localization contributions therefore requires a different approach.\\

 The numerical consequence of Theorem~\ref{thm:main1} requires both genus zero and positivity of $Z$. If either assumption is dropped, we provide the following counterexamples:
 \begin{itemize}
     \item [(a)] If $g=0$ but $Z$ is not positive: see Example~\ref{exam:counterex for non-positive}.
     \item[(b)] If $g>0$: see Remark~\ref{rmk:counterex for case ii} ($Z$ is positive in this counterexample). 
 \end{itemize}
 Moreover, we disprove Conjecture~\ref{conj:Hu} in Section~\ref{sec:counterex of Hu's conj}. The counterexample shows that the condition $g\leq 1$ in case (I) is sharp, and the condition $d\leq3$ in case (II) cannot be omitted.

\subsection{Acknowledgements}
The authors thank Jeongseok Oh for many valuable discussions on this subject over years. Sanghyeon Lee thanks Yuan-Pin Lee for introducing him to various problems concerning genus zero Gromov--Witten theory of blow-ups in late 2023.
Seungjae Yun thanks Jeongseok Oh for introducing him to this problem, and for many stimulating conversations.

\section{Proof of the blow-up formula in genus zero}
\label{sec: proof}
Let $W \coloneq \Bl_{Z \times \set{0}} (X \times \PP^1)$, where $0:= [1:0] \in \PP^1$ and $\infty:= [0:1] \in \PP^1$. For an effective curve class $\beta \in \rH_2(X;\Z)$, we define the master space as a moduli space of stable maps to $W$. In genus zero, this master space was used in \cite{Iri23} to prove a decomposition theorem for quantum cohomology of blow-ups. It is also used in \cite{FL19} to obtain an algorithm for computing higher-genus Gromov--Witten invariants of smooth hypersurfaces.

Consider a $\C^*$-action on $\PP^1$ given by $\lambda \cdot [a:b] := [\lambda a : b]$ for $\lambda \in \C^*$ and $[a:b] \in \PP^1$. This naturally induces $\C^*$-actions on $X \times \PP^1$ and $W$. The fixed locus of $W$ is given by
\begin{equation*}
    W^{\C^*} = X_\infty \sqcup Z_* \sqcup \wtil X , 
\end{equation*}
where $X_\infty \subset W \times_{\PP^1} \set{\infty}$, $\wtil X \subset W \times_{\PP^1} \set{0}$, and $Z_* = \PP_Z(0 \oplus 1) \subset \what{E}$, where $\what{E} \cong \PP_Z(N_{Z/X} \oplus 1)$ is the exceptional divisor. Note that $X_\infty$ is naturally isomorphic to $X$ and $Z_*$ is naturally isomorphic to $Z$. Let $E$ be the exceptional divisor of $\wtil X$. Figure~\ref{figure 1} illustrates the geometry.

\begin{figure}[h]
\tikzset{every picture/.style={line width=0.75pt}} 
\begin{tikzpicture}[x=0.75pt,y=0.75pt,yscale=-1,xscale=1]

\draw   (114.35,115.98) -- (114.35,22.77) -- (170.31,62.72) -- (170.31,155.92) -- cycle ;
\draw   (258.35,108.12) -- (258.35,13.77) -- (314.31,52.57) -- (314.31,146.92) -- cycle ;
\draw   (129.33,95.5) -- (129.33,64.1) -- (155.33,83.19) -- (155.33,114.6) -- cycle ;
\draw [line width=1.5]    (202.33,76.5) -- (228.33,95.6) ;
\draw    (155.33,83.19) -- (228.33,95.6) ;
\draw    (129.33,64.1) -- (202.33,76.5) ;
\draw  [dash pattern={on 0.84pt off 2.51pt}]  (129.33,95.5) -- (202.33,76.5) ;
\draw    (155.33,114.6) -- (228.33,95.6) ;

\draw (174,149) node [anchor=north west][inner sep=0.75pt]   [align=left] {0};
\draw (321,145) node [anchor=north west][inner sep=0.75pt]   [align=left] {$\displaystyle \infty $};
\draw (126,138) node [anchor=north west][inner sep=0.75pt]   [align=left] {$\displaystyle \widetilde{X}$};
\draw (285,139) node [anchor=north west][inner sep=0.75pt]   [align=left] {$\displaystyle X_\infty$};
\draw (183,53) node [anchor=north west][inner sep=0.75pt]  [font=\footnotesize] [align=left] {$Z_*$};
\draw (134,83) node [anchor=north west][inner sep=0.75pt]  [font=\footnotesize] [align=left] {$\displaystyle E$};
\draw (133,71) node [anchor=north west][inner sep=0.75pt]  [font=\tiny] [align=left] {};
\draw (133,104) node [anchor=north west][inner sep=0.75pt]  [font=\tiny] [align=left] {};
\draw (216,79) node [anchor=north west][inner sep=0.75pt]  [font=\tiny] [align=left] {};
\draw (187,109) node [anchor=north west][inner sep=0.75pt]  [font=\tiny] [align=left] {};
\draw (156,95) node [anchor=north west][inner sep=0.75pt]  [font=\tiny] [align=left] {};
\draw (122,78) node [anchor=north west][inner sep=0.75pt]  [font=\tiny] [align=left] {};
\draw (181,89) node [anchor=north west][inner sep=0.75pt]  [font=\tiny] [align=left] {};

\end{tikzpicture}
\caption{The structure of $W$}
\label{figure 1}
\end{figure}
This naturally induces a $\C^*$-action on $\overline\cM(W)$. We specify the curve class in Section~\ref{sec: localization}. Note that we have
\begin{equation*}
    \rH_2(W;\Z) = \rH_2(X \times \PP^1;\Z) \oplus \Z[\ell]
\end{equation*}
where $\ell$ is the class of a line in a fiber of the projective bundle $E$. Let $L$ be the class of $\PP^1$, which is the second factor of the product $X \times \PP^1$. Then we have
\begin{equation*}
\rH_2(W;\Z) = \rH_2(X; \Z) \oplus \Z[\ell] \oplus \Z[L].
\end{equation*}
By abuse of notation, we will denote all three morphisms $W \to X \times \PP^1$, $E\to Z$ and $\wtil{X} \to X$ by $\pi$.
\subsection{Localization of the master space}\label{sec: localization}

Consider the $\C^*$-action on the master space $\overline\cM(W)\coloneqq \overline\cM(W,\beta')$ described above. We consider the following two cases for its degree $\beta'$:
\begin{enumerate}
\item
$\beta' = \pi^! \beta$ for $\beta \in \rH_2(X; \Z)$
\item
$\beta' = \pi^! \beta + M \ell$ for $\beta \in \rH_2(X; \Z)$ and $M > 0$.
\end{enumerate}

First, consider case (1). Because there is no coefficient of $L$ in the curve class $\beta'$, the $\C^*$-fixed loci of $\overline\cM(W)$ are classified as follows. Let $\iota:Z\into X$ be the inclusion.\\

\textbf{Type~I.}
The fixed locus lies over $\infty\in\P^1$.
\begin{equation*}
    \overline\cM_{0,k}(X_\infty, \beta ) \cong \overline\cM_{0,k}(X, \beta )
\end{equation*}

\textbf{Type~II.}
This locus parametrizes maps whose images are contained in $\wtil X$ lying over $0\in\P^1$.
\begin{equation*}
    \overline\cM_{0,k}(\Bl_{Z \times \set{0}} (X \times \set{0}), \pi^!\beta ) \cong \overline\cM_{0,k}(\wtil{X},\pi^! \beta).
\end{equation*}

\textbf{Type~III.}
This locus parametrizes maps whose images are contained in $Z_*\subset \what E$, lying over $0\in\P^1$.
\begin{equation*}
    \overline\cM_{0,k}(Z_*, \beta_Z) \cong \overline\cM_{0,k}(Z, \beta_Z)
\end{equation*}
 where $\beta_Z$ satisfies $\iota_*\beta_Z=\beta$ in $\rH_2(X;\Z)$.\\

\textbf{Type~IV.}
 Other fixed components correspond to decorated graphs with at least one edge, as in \cite{Kon95}. We let  $\overline\cM_{\Gamma}$ be the space corresponding to each localization graph $\Gamma$. A decorated graph $\Gamma$ consists of the following data: a partition, say $\Ver(\Gamma) = V_0 \sqcup V_*$ with a set of legs $\Leg(\Gamma)$ attached to the vertices. The vertices in $V_0$ and $V_*$ correspond to connected fixed subcurves of the domain curve that are mapped to $\wtil X$ and $Z$, respectively. They are allowed to be points. Each edge $e \in \Edg(\Gamma)$ joins a vertex $v \in V_0$ to a vertex $u \in V_*$; its covering degree is denoted by $d_e$. All genus weights of the graph are zero, and each vertex and edge is decorated by degrees $\beta_{v}, \beta_{u}, \beta_{e} \in \rH_2(W;\Z)$. The flags at a vertex are its legs and incident edges; denote their set by $F(v)$ for $v\in \Ver(\Gamma)$. A decorated graph $\Gamma$ corresponds to a $\C^*$-fixed locus $\overline\cM_\Gamma$ parameterizing stable maps $f : (C,x_1,\cdots,x_k) \to W \times_{\PP^1} \set{0}$ satisfying the following conditions:
\begin{itemize}
\item[(a)]
The genus-zero domain curve $C$ is decomposed as
\begin{equation*}
C = \left( \bigcup_{u \in V_*} C_u \right) \cup \left( \bigcup_{v \in V_0} C_v \right) \cup \left( \bigcup_{e \in \Edg} C_e \right),
\end{equation*}
such that each $C_e \cong \PP^1$ connects a component corresponding to a vertex in $V_*$ and a component corresponding to a vertex in $V_0$. Note that we allow $C_v,C_u$ to be a point. 
\item[(b)]
$f({C_u})$ is contained in $\PP_Z(0 \oplus 1) \cong Z$ for each $u\in V_*$, and $f({C_v})$ is contained in $\Bl_{Z\times \set{0}} (X\times \set{0} ) \cong \wtil{X}$ for each $v\in V_0$. The restriction $f|_{C_e}$ is a ramified cover of a line by $\PP^1 \cong C_e$ joining a point $x \in \PP(N_{Z/X}\oplus 0) \subset \wtil{X}$ and the point $\pi(x) \in \PP_Z(0 \oplus 1) \cong Z$. Here, $x$ and $\pi(x)$ are the two fixed endpoints.  
\end{itemize}
A direct computation of virtual dimension shows the following.
\begin{equation}
\label{eq:vdim1}
\vd\overline\cM_{0,k}(X,\beta) = \vd\overline\cM_{0,k}(\wtil{X},\pi^!\beta) = \vd\overline\cM(W) - 1. 
\end{equation}
Therefore, the virtual dimensions of Type I, II fixed loci are the same as the virtual dimension of the master space $\overline\cM(W)$ minus $1$.

On the other hand, we have
\begin{align}\label{eq:vdim3}
\vd\overline\cM_{0,k}(Z,\beta_Z) & = (\dim Z - 3) + k - c_1(K_Z)\cdot \beta_Z \\ \nonumber
& = \vd\overline\cM_{0,k}(X,\beta) - ( r + c_1(N_{Z/X} ) \cdot \beta_Z) 
\end{align}
for $\iota_*\beta_Z=\beta$. Here we recall that $r$ is the codimension of $Z \subset X$. Now we introduce the following assumption on $Z$, as in \cite[Definition 1.1]{CD23}.

\begin{defn}
\label{defn:positivity}
We say that the blow-up center $Z \subset X$ is positive if it satisfies $c_1(N_{Z/X}) \cdot f_*[\PP^1] +r > 0$ for any regular map $f : \PP^1 \to Z$.
\end{defn}
\begin{rmk}
\label{rmk: alternative of positivity}
    Positivity of $Z$ is equivalent to assuming $c_1(N_{Z/X}) \cdot f_*[\PP^1]\geq0$ for any regular map $f : \PP^1 \to Z$. If a regular map $f$ had $c_1(N_{Z/X}) \cdot f_*[\PP^1]<0$, a sufficiently high multiple cover of it would contradict the assumption.
\end{rmk}

Assuming positivity of $Z$, the virtual dimension of Type~III fixed loci is strictly smaller than the virtual dimension of Type~I, II loci. We now analyze the contributions of Type~IV loci.\\

Given a localization graph $\Gamma$, recall the data of $V_0$, $V_*$, and $\Edg$ defined above. For $v\in V_0$ and $u\in V_*$, set
\begin{equation*}
    \overline{\mathcal M}_v=\overline{\mathcal M}_{0,n(v)}
    (\wtil X,\beta_v),
    \qquad
    \overline{\mathcal M}_u=\overline{\mathcal M}_{0,n(u)}(Z,\beta_u),
\end{equation*}
where the markings corresponding to the incident edges are included among
the $n(v)$ and $n(u)$ markings. For unstable vertices, we use the convention $\overline\cM_v=\wtil X$ and $\overline\cM_u=Z$. For an edge $e$, let
$\overline{\mathcal M}_e$ denote the fixed component parameterizing the $\C^*$-invariant degree-$d_e$ covers, $f\vert_{C_e}:C_e \to \ell_e$. Its coarse moduli space is naturally isomorphic to $E$. It has endpoint evaluation maps
\begin{equation*}
    \ev_e^-:\overline{\mathcal M}_e\longrightarrow  \wtil X,\qquad
    \ev_e^+:\overline{\mathcal M}_e\longrightarrow Z,
\end{equation*}
where $\ev_e^-$ factors through $E\hookrightarrow \wtil X$ and
$\ev_e^+$ is the composition with $E\to Z$.

Set
\begin{equation*}
    \mathcal V_0=\prod_{v\in V_0}\overline{\mathcal M}_v,
    \qquad
    \mathcal V_*=\prod_{u\in V_*}\overline{\mathcal M}_u,
    \qquad
    \mathcal E_\Gamma=\prod_{e\in \Edg}\overline{\mathcal M}_e .
\end{equation*}
The endpoint evaluations define the morphisms
\begin{equation*}
    \ev_{\mathrm{vert}}:\mathcal V_0\times\mathcal V_*
    \longrightarrow \wtil X^{\Edg}\times Z^{\Edg},\quad
    \ev_{\mathrm{edge}}:\mathcal E_\Gamma\longrightarrow  \wtil X^{\Edg}\times Z^{\Edg}.
\end{equation*}
Then the fixed locus associated to $\Gamma$ is naturally the fiber product
\begin{equation}
\label{eq: moduli level iso}
    \overline{\mathcal M}_\Gamma\cong(\mathcal V_0\times\mathcal V_*) \times_{\wtil X^{\Edg}\times Z^{\Edg}}\mathcal E_\Gamma,
\end{equation}
as in \cite[Section 2.3]{CD23}. Here we are working temporarily with a labeled localization graph $\Gamma$, so that graph-automorphism factors are suppressed. The actual fixed moduli stack is the quotient by $\Aut(\Gamma)$ and has multiple cover factors, which do not alter the obstruction theories. This convention is only used until Lemma~\ref{lem:fixed locus gysin} below. Now we show that the perfect obstruction theory on $\overline{\mathcal M}_\Gamma$ induced
by the fixed part of the perfect obstruction theory of the master space
coincides with the fiber-product perfect obstruction theory. Let
\begin{equation*}
    \Delta_\Gamma: \wtil X^{\Edg}\times Z^{\Edg}
    \longrightarrow\bigl( \wtil X^{\Edg}\times Z^{\Edg}
    \bigr)^2
\end{equation*}
be the diagonal.

\begin{lem} Under the identification \eqref{eq: moduli level iso}, the virtual cycles satisfy:
\label{lem:fixed locus gysin}
\begin{equation}
\label{eq:fixed locus gysin}
    [\overline{\mathcal M}_\Gamma]^{\vir}=\Delta_\Gamma^!
    \left(\bigboxtimes_{v\in V_0}[\overline{\mathcal M}_v]^{\vir}
        \boxtimes\bigboxtimes_{u\in V_*}[\overline{\mathcal M}_u]^{\vir}
        \boxtimes\bigboxtimes_{e\in \Edg}[\overline{\mathcal M}_e]
    \right)
\end{equation}
in $A_*(\overline{\mathcal M}_\Gamma)$, under the identification \eqref{eq: moduli level iso}. Equivalently, writing
\begin{equation*}
    \Delta_\Gamma=\Delta_{ \wtil X^{\Edg}}
    \times\Delta_{Z^{\Edg}},
\end{equation*}
the virtual class is obtained by imposing the matching conditions at the
$\wtil X$- and $Z$-ends of the edges by the corresponding refined
Gysin pullbacks.
\end{lem}

\begin{proof}
Let $f:C\longrightarrow W$ be a stable map corresponding to a geometric point of $\overline{\mathcal M}_\Gamma$.
Normalizing the source curve at all nodes lying between an edge component and a vertex component gives
\begin{equation*}
    C^{\nu}= \coprod_{v\in V_0} C_v \sqcup \coprod_{u\in V_*} C_u
    \sqcup \coprod_{e\in \Edg} C_e .
\end{equation*}
The restriction of $f$ to $C_v$ lies in $\wtil X$, the restriction
to $C_u$ lies in $Z$, and the restriction to $C_e$ is the corresponding
invariant cover. Let $\mathrm{Nodes}(\Gamma)$ be the set of attachment nodes in this normalization.

We first compare the deformation complexes. Write $\rho$ for the universal-curve projection, and use subscripts for the corresponding projections on vertex and edge factors. The standard normalization sequence
gives a distinguished triangle
\begin{align}
R\rho_*f^*T_W\longrightarrow&\bigoplus_{v\in V_0}R\rho_{v*}f_v^*T_W\oplus\bigoplus_{u\in V_*}R\rho_{u*}f_u^*T_W\oplus\bigoplus_{e\in \Edg}R\rho_{e*}f_e^*T_W
\nonumber\\
&\longrightarrow
\bigoplus_{\mathfrak n\in\mathrm{Nodes}(\Gamma)}T_W|_{f(\mathfrak n)}\xrightarrow{+1}.
\label{eq:nor triangle}
\end{align}
We now take the $\mathbb C^*$-fixed part. At a node $\mathfrak n$ joining an edge component to a $\wtil X$-vertex component,  $\bigl(T_W|_{f(\mathfrak n)}\bigr)^{\fix}=T_{
\wtil X}|_{f(\mathfrak n)}$.
At a node $\mathfrak n $ joining an edge component to a $*$-vertex component, $\bigl(T_W|_{f(\mathfrak n)}\bigr)^{\fix}=T_Z|_{f(\mathfrak n)}$.
Therefore, the fixed part of the last term of
\eqref{eq:nor triangle} is
\begin{equation*}
    \bigoplus_{e\in \Edg}
    \left(T_{\wtil X}|_{\ev_e^-}\oplus T_Z|_{\ev_e^+}
    \right).
\end{equation*}
Taking fixed parts of \eqref{eq:nor triangle} gives
\begin{align}
\left(R\rho_*f^*T_W\right)^{\fix}\longrightarrow {}&\bigoplus_{v\in V_0}R\rho_{v*}f_v^*T_{\wtil X}
\oplus\bigoplus_{u\in V_*}R\rho_{u*}f_u^*T_Z
\oplus\bigoplus_{e\in \Edg}\left(R\rho_{e*}f_e^*T_W\right)^{\fix}
\nonumber\\
&\longrightarrow\bigoplus_{e\in \Edg}\left(T_{\wtil X}|_{\ev_e^-}\oplus T_Z|_{\ev_e^+}\right)
\xrightarrow{+1}.
\label{eq:fixed nor triangle}
\end{align}
Let $\ell$ be the image of $C_e$; then the factor $(R\rho_{e*}f_e^*N_{\ell/W})^{\fix}$ is canonically identified with $(\ev_{e}^-)^*T_E \oplus R\Gamma(C_e, f_e^* N_{\what{E}/W})^{\fix}$.
Since $f_e^* N_{\what{E}/W} \cong \cO_{\PP^1}(-d_e)$, we get $H^0(C_e, f_e^* N_{\what E/W})=0$ while $H^1(C_e, f_e^* N_{\what E/W}) \cong H^0(C_e, \cO_{C_e}(d_e - 2))^\vee$ is in the moving part. The extra summand $\C \cong (R\rho_{e*}f_e^*T_{\ell})^{\fix}$ in $\left(R\rho_{e*}f_e^*T_W\right)^{\fix}$ cancels with the infinitesimal automorphism term for the edge. For the absolute dual perfect obstruction theory, the smoothing parameters of the nodes joining an invariant edge to a vertex have nonzero $\mathbb C^*$-weights and hence are in the moving part. Note that $\cE_\Gamma$ is smooth and $T_{\cE_\Gamma}$ is naturally isomorphic to $(\ev_{e}^-)^*T_E$. 
On the other hand, we can consider the natural virtual tangent complex of $\overline{\cM}_\Gamma \cong (\mathcal V_0\times\mathcal V_*)
    \times_{\wtil X^{\Edg}\times Z^{\Edg}}
    \mathcal E_\Gamma$ defined by the cone:
\begin{equation*}
    \TT_{\overline{\cM}_\Gamma} :=\Cone\left(\mathbb T^{\vir}_{\mathcal V_0\times\mathcal V_*}\oplus T_{\mathcal E_\Gamma}\longrightarrow T_{\wtil X^{\Edg}\times Z^{\Edg}}\right)[-1].
\end{equation*}

The preceding computation identifies the fixed part of the restricted virtual tangent complex with $\mathbb{T}_{\overline{\cM}_\Gamma}$. This proves \eqref{eq:fixed locus gysin}.

Note that the above argument also covers the case with unstable vertices. For a flag $(v,e)$, we have $\Delta^!([\overline\cM_v]\boxtimes[\cM_e])=[\cM_e]$ for univalent $v$ and $(\Delta\times\Delta)^!(\alpha_1\boxtimes[\overline\cM_v]\boxtimes\alpha_2)=\Delta^!(\alpha_1\boxtimes\alpha_2)$ for bivalent $v$.
\end{proof}

This lemma allows us to compute the virtual dimension of Type~IV fixed loci. First, let $U_1$ be the number of univalent unstable vertices and $U_2$ be the number of bivalent unstable vertices. The valence counts markings as well as incident edges. Let $n=|V_0|$, $m=|V_*|$, $h=|\Edg|$, and $\mathfrak e$ ($\geq h$) be the total degree of the edges. Then, $h=n+m-1$ since the genus is zero. Let $\beta_*=\sum_{u\in V_*}\beta_u$ be the total degree of $*$-vertices and denote $$\delta_\Gamma=c_1(N_{Z/X})\cdot \beta_*+r.$$

\begin{lem}\label{lemm:vdim1}
The virtual codimension of Type~IV fixed loci is given by:
\begin{equation*}
    \vcd \overline\cM_\Gamma-1=\delta_\Gamma +r(m-1)+h+(r-1)(\mathfrak e-h)-2U_1-U_2.
\end{equation*}
In particular, positivity of $Z$ implies $\delta_\Gamma>0$, and we obtain
\begin{equation*}
    \vcd \overline\cM_\Gamma-1> r(m-1)+h-2U_1-U_2.
\end{equation*}
\end{lem}
\begin{proof}
Let $\mathbb E$ be the perfect obstruction theory of $\overline\cM(W)$  relative to the Artin stack of prestable curves with marked points. Let $\Gamma$ be a localization graph with $v\in V_0$ decorated by degree $$\beta_v = \pi^!\beta(v) - \mathfrak e_v \cdot \ell$$ for some $\beta(v) \in \rH_2(X;\Z)$ and $\mathfrak e_v\in \Z$, and $u\in V_*$ decorated by degree $\beta_u \in \rH_2(Z;\Z)$. The curve classes satisfy
\begin{equation*}
   \iota_*(\beta_*)+\sum_{v\in V_0}\beta (v)= \beta\quad\text{and}\quad \sum_{v\in V_0}\mathfrak e_v=\mathfrak e.
\end{equation*}
The contribution of $\bigcup_{u \in V_*} C_u $ to the virtual dimension is
    \begin{equation*}
        \sum_{u \in V_*} \left(\vd \overline\cM_{0,\val(u)}+\rk \mathbb E\vert_{C_{u}}^{\fix}\right)=\sum_{u \in V_*}\vd \overline\cM_{0,\val(u)}(Z,\beta_u).
    \end{equation*}
Similarly, the contribution of $\bigcup_{v \in V_0} C_v $ is $\sum_{v \in V_0}\vd\overline\cM_{0,\val(v)}(\wtil X,\beta_v)$. The contribution of $\bigcup_{e \in \Edg} C_e $ to the virtual dimension is $h\cdot (d-1)$. The node-smoothing parameters have nonzero $\C^*$-weights, 
so we can divide the virtual dimension into contributions of $C_u$, $C_v$, and $C_{e}$. Note that there are matching conditions imposed at the nodes that contribute $(-\dim Z)$ at each edge endpoint attached to a vertex $u\in V_*$ and $(-\dim \wtil X)$ at each edge endpoint attached to $v\in V_0$. Hence, the virtual dimension of $\overline\cM_\Gamma$ is
\begin{align*}
\vd \overline\cM_\Gamma &=n\cdot (d-3) + \sum_{v\in V_0} \left(\int_{\beta(v)} c_1(X) - (r-1)\mathfrak e_v\right) \\
&+ m\cdot(d-r-3) + \left( \sum_{u\in V_*} \int_{\beta_{u}} c_1(Z) \right) -(d-r-1)\cdot h+k+2U_1+U_2\\
&= \vd\overline\cM_{0,k}(X,\beta)-\delta_\Gamma-r(m-1)-h-(r-1)(\mathfrak e-h)+2U_1+U_2
\end{align*}
where we used $h=n+m-1$ in the last equality.
\end{proof}

Let $\Omega$ be the set of localization graphs of Types~III or IV. From the virtual localization formula of \cite{GP99}, we get
\begin{multline}
\label{eq:localization}
    [\overline\cM(W)]^{\vir}= \frac{\iota_*[\overline\cM_{0,k}(X,\beta)]^{\vir}}{e^{\C^*}(N^{\vir}|_{\overline\cM_{0,k}(X,\beta)})} +\frac{\iota_*[\overline\cM_{0,k}(\wtil{X},\pi^!\beta)]^{\vir}}{e^{\C^*}(N^{\vir} |_{\overline\cM_{0,k}(\wtil{X},\pi^!\beta)} )}\\
    +\sum_{\Gamma \in\Omega}\frac{1}{|\Aut(\Gamma)|}\cdot \frac{\iota_*[\overline\cM_\Gamma]^{\vir}}{e^{\C^*}(N^{\vir} \vert_{\overline\cM_\Gamma})}
\end{multline}
in $A_{*}^{\C^*}(\overline\cM(W))\otimes_{\Q[t]}\Q(t)$ where $\iota_*$ denotes the equivariant pushforward from the corresponding fixed locus to $\overline\cM(W)$. We now compute Euler classes of the virtual normal bundles as
\begin{align*}
N^{\vir} |_{\overline\cM_{0,k}(X,\beta)} = R\rho_*(f^* N_{X \times \set{\infty} / W}) = R \rho_*(\cO_{\cC}) \cong [\cO_{\overline\cM_{0,k}(X,\beta)} \otimes \mathbf{t} \to 0]
\end{align*}
where $\rho : \cC \to \overline\cM_{0,k}(X,\beta)$ is the universal curve and $f : \cC \to X \times \set{\infty}$ is the universal morphism; $\mathbf{t}$ is a trivial bundle with weight $1$ via the $\C^*$-action. Thus, we have
\begin{align*}
\frac{1}{e^{\C^*}(N^{\vir}|_{\overline\cM_{0,k}(X,\beta)})} = \frac{1}{t}.
\end{align*}
Moreover, we have $N_{\wtil{X}/W} \cong \cO(-E) \otimes \mathbf{t}^{-1}$ so that
\begin{align*}
\frac{1}{e^{\C^*}(N^{\vir} |_{\overline\cM_{0,k}(\wtil{X},\pi^!\beta)} ) } = \frac{-1}{t} + O(t^{-2}).
\end{align*}
For a Type~IV graph with $\vcd\overline\cM_\Gamma>1$, 
\begin{align*}
\frac{1}{e^{\C^*}(N^{\vir} |_{\overline\cM_\Gamma} ) } =  O(t^{-2}).
\end{align*}
Now, let $p$ be the composite projection $W\to X\times \PP^1\to X$. We also write $p: \overline\cM(W)\to \overline\cM_{0,k}(X,\beta)$ to be the induced morphism between the moduli spaces. Then consider the cycle $$p_*[\overline\cM(W)]^{\vir}\in A_{*}(\overline\cM_{0,k}(X,\beta))$$ whose $\C^*$-equivariant lift is given by taking the equivariant pushforward $p_*$ of \eqref{eq:localization}. Once we show that Type~IV graphs with $\vcd\overline\cM_\Gamma\leq1$ do not contribute to the residue, taking the residue of $p_*[\overline\cM(W)]^{\vir}$ will prove our main theorem (Theorem~\ref{thm:main1}). Hence, our next step is to compute the contributions of Type~IV graphs with $\vcd\overline\cM_\Gamma\leq1$.

\subsection{Type~IV loci with small virtual codimension}
\label{sec: Type IV vanishing}
Throughout this subsection, we assume $a_\Gamma\coloneqq \vcd \overline\cM_\Gamma\leq1$. Let 
\begin{equation*}
    S=n+m-U_1-U_2
\end{equation*}
be the number of stable vertices, before the contractions and forgetting procedures described below. By Lemma~\ref{lemm:vdim1}, we get
\begin{equation*}
    a_\Gamma-1>(S-1)-U_1
\end{equation*}
since $m\geq 1$ and $h=n+m-1$. In particular, we have $U_1\geq1$. Let $V_i^{\un}(\Gamma)$ be the set of vertices decorated with genus zero and degree zero, with $i$ flags attached. Here, flags count both the legs and the incident edges. For $u\in V_1^{\un}(\Gamma)$, its unique incident edge has a free ramification point, denoted $a_u$. Denote
\begin{equation*}
    U_i=|V_i^{\un}(\Gamma)|,\quad A_\Gamma=\{a_u\mid u\in V_1^{\un}(\Gamma)\}.
\end{equation*}
In particular, $|A_\Gamma|=U_1$. Each $a_u$ can be seen as a section of the universal curve over the fixed moduli $\overline\cM_\Gamma$. Now, let $\rho:\mathcal C\to \overline\cM_\Gamma$ be the universal curve and $f:\cC\to W$ be the universal map of $\overline\cM_\Gamma$. Compose the universal map $f$ with $p:W\to X$, and retain the marking sections. The resulting map on a geometric fiber of this universal curve may not be stable, since contracting a nonzero-degree component to a point may make it unstable. After further stabilization, denote the graph by $\what \Gamma$. It has $k+U_1$ legs. The procedure above gives a contraction morphism
\begin{equation*}
    \Phi_\Gamma:\overline\cM_\Gamma\longrightarrow \overline\cM_{\what \Gamma}(X).
\end{equation*}
Here, $\overline\cM_{\what \Gamma}(X)$ can be viewed as a boundary stratum in $ \overline\cM_{0,k+U_1}(X)$ through a gluing morphism. The graph $\what \Gamma$ is obtained by the following operations.
\begin{itemize}
    \item [(1)] A maximal chain of edges (which are contracted by $p$) connecting two stable vertices becomes a nodal edge. During the process, an unstable bivalent vertex with two edges is removed from the chain.
    \item [(2)] A chain ending at $u\in V_1^{\un}(\Gamma)$ becomes a leg $a_u$.
    \item [(3)] A bivalent unstable vertex with one edge and one leg transfers the leg to its unique adjacent vertex.
    \item [(4)] Apply $p_*$ to every vertex degree.
    \item [(5)] Stabilize all genus-zero, degree-zero vertices with at most two flags.
\end{itemize}
In particular, we have the set $A_\Gamma$ of legs of $\what\Gamma$. We can forget these legs (markings) and stabilize once more:
\begin{equation*}
    \overline\Gamma=\st_X(\what \Gamma\setminus A_\Gamma).
\end{equation*}
Vertices that become unstable after discarding $A_\Gamma$ are contracted in the process. Now, for any stable decorated graph $G$, we can define its associated boundary stratum in $\overline\cM(X)$ as
\begin{equation*}
    \overline \cM_G(X)= \left(\prod_{v\in \Ver(G)}\overline\cM_{0,F(v)}(X,\beta(v))\right)\times_{X^{2\cdot |\Edg(G)|}}X^{|\Edg(G)|}.
\end{equation*}
The fibered product is taken via the evaluation maps $\ev:\overline\cM_{0,F(v)}(X,\beta(v))\to X$ at markings indexed by edge flags, and the diagonal map $$\Delta:X^{|\Edg(G)|}\to X^{2\cdot |\Edg(G)|}.$$
Choose an ordering $A_\Gamma=\{a_1,\cdots, a_{U_1}\}$ and define inductively
\begin{equation*}
    G_0=\what \Gamma,\quad G_i=\st_X(G_{i-1}\setminus\{a_i\}).
\end{equation*}
Let $\rho_i:\overline\cM_{G_{i-1}}(X)\to \overline\cM_{G_{i}}(X)$ be the natural stabilization morphism. Returning to our fixed locus, consider the composition of the stabilization morphisms:
\begin{equation*}
    \rho_\Gamma=\rho_{U_1}\circ\cdots\circ \rho_1:\overline\cM_{\what \Gamma}(X)\longrightarrow\overline\cM_{\overline \Gamma}(X).
\end{equation*}
Then, the composition $\overline\cM_{\Gamma}\into \overline\cM_{0,k}(W)\xrightarrow[]{p}\overline\cM_{0,k}(X)$ can be factored as
\begin{equation}
\label{eq: factorization}
    \overline\cM_{\Gamma}\xrightarrow[]{\Phi_\Gamma} \overline\cM_{\what \Gamma}(X)\xrightarrow[]{\rho_\Gamma}\overline\cM_{\overline \Gamma}(X)\xrightarrow[]{\gl_{\overline\Gamma}}\overline\cM_{0,k}(X)
\end{equation}
where $\gl_{\overline \Gamma}$ is the gluing morphism associated with $\overline\Gamma$. One can show that the composition \eqref{eq: factorization} is independent of the ordering of $A_\Gamma$. Now let $U_1^0$ be the number of univalent unstable $0$-vertices and $U_1^*$ be the number of univalent unstable $*$-vertices so that 
\begin{equation*}
    U_1=U_1^0+U_1^*.
\end{equation*}
We further factor $\Phi_\Gamma$. First of all, let $$\mathfrak b: \what \cM_\Gamma\longrightarrow\overline\cM_\Gamma$$ be a finite cover of $\overline\cM_\Gamma$, obtained by labeling every univalent unstable $0$-vertex. In particular, we have $\mathfrak b_*[\what{\cM}_\Gamma]^{\vir} = (\deg\mathfrak b)\cdot [\overline{\cM}_\Gamma]^{\vir}$.

Then, we define $\wtil \Gamma$ to be the graph obtained from $\Gamma$ by replacing each unstable univalent $0$-vertex and its unique edge by an additional marking at the unique adjacent vertex. On the corresponding moduli spaces, the edge factor $\overline\cM_e$ is contracted via its evaluation map $\ev_e^+:\overline\cM_e\longrightarrow Z$, using Lemma~\ref{lem:fixed locus gysin}, leaving an additional marking. Let us define $\overline\cM_{\wtil \Gamma}$ analogously to \eqref{eq:fixed locus gysin}. As a result, we obtain the natural morphism
\begin{equation*}
    \omega_\Gamma:\what\cM_\Gamma\longrightarrow\overline\cM_{\wtil \Gamma}.
\end{equation*}
Then $\omega_\Gamma$ factors the composition $\what{\cM}_\Gamma \to \overline{\cM}_\Gamma \xrightarrow[]{\Phi_\Gamma} \overline{\cM}_{\what{\Gamma}}(X)$. 
Note that $\omega_\Gamma$ is an iterated gerbe over projective bundles with total relative virtual dimension $(r-1)\cdot U_1^0$. In particular, by Lemma~\ref{lem:fixed locus gysin}, $\omega_\Gamma$ is a virtually smooth morphism (see \cite[Definition 3.4]{Man12b} for the definition). Then, by the virtual pullback \cite[Theorem 1]{KKP03} or \cite{Man12a}, we have $[\what \cM_{\Gamma}]^{\vir}=(\omega_\Gamma)^! [\overline\cM_{\wtil \Gamma}]^{\vir}$.
It suffices to show
\begin{equation}
    \label{eq: contribution vanishes}(\omega_\Gamma)_*\left([t^{-1}]\frac{[\what\cM_\Gamma]^{\vir}}{e^{\C^*}(\mathfrak b^*N_\Gamma^{\vir})}\right)=0.
\end{equation}

\begin{lem}
\label{lem: contribution vanishes}
    If $Z$ is positive, equation \eqref{eq: contribution vanishes} holds.
\end{lem}
\begin{proof}
    The proof uses only a comparison of virtual dimensions. Recall that we are assuming the virtual codimension $a_\Gamma$ of $\overline \cM_\Gamma$  satisfies $a_\Gamma\leq1$. Hence, the coefficient appearing in \eqref{eq: contribution vanishes} is the proper pushforward of $P_\Gamma\cap[\what\cM_\Gamma]^{\vir}$ along $\omega_\Gamma$ where $P_\Gamma$ is a degree-$(1-a_\Gamma)$ element of the operational Chow group. Now Lemma~\ref{lemm:vdim1} shows
    \begin{equation*}
        a_\Gamma-1+U_1^0=\delta_\Gamma+r(m-1)+(S-1)+(r-1)(\mathfrak e-h)-U_1^*.
    \end{equation*}
    If $m>U_1^*$, then the right-hand side is positive, since $\delta_\Gamma>0$. If $m=U_1^*$, then all $*$-vertices are degree-zero leaves, hence $\beta_*=0$ and $\delta_\Gamma=r$. In either case, we have $a_\Gamma-1+U_1^0>0$. Hence, 
    \begin{equation} \label{eq:codimcond1}
        1-a_\Gamma<U_1^0\leq (r-1)\cdot U_1^0=\rel.\vd \omega_\Gamma.
    \end{equation}
    Since $\omega_\Gamma$ is virtually smooth, the virtual push-forward property \cite[Definition~3.4, Lemma~3.6]{Man12b} and \eqref{eq:codimcond1} imply that $(\omega_\Gamma)_* \left( P_\Gamma \cap [\what\cM_{\Gamma}]^{\vir} \right)$ vanishes.
    Thus, we obtain the vanishing \eqref{eq: contribution vanishes}.
\end{proof}
This lemma completes the proof of our first main theorem.
\begin{thm}\label{thm:main1}
Assuming positivity of the blow-up center $Z$, we have
\begin{equation*}
\pi_*[\overline\cM_{0,k}(\wtil X,\pi^!\beta)]^{\vir}=[\overline\cM_{0,k}(X,\beta)]^{\vir}
\end{equation*}
in $A_*(\overline\cM_{0,k}(X,\beta))$. In particular, we have 
\begin{equation*}
\langle \pi^*\gamma_1,\cdots, \pi^*\gamma_k\rangle^{\wtil{X}}_{0,\pi^!\beta}=\langle \gamma_1,\cdots, \gamma_k\rangle^X_{0,\beta}
\end{equation*}
for $\gamma_i\in \rH^*(X;\Z)$.  
\end{thm}
\begin{proof}
    Recall the localization formula \eqref{eq:localization}. We have shown that Type~III graphs and Type~IV graphs with $\vcd\overline\cM_\Gamma>1$ do not contribute to the residue. Moreover, Lemma~\ref{lem: contribution vanishes} shows that Type~IV loci with $\vcd\overline\cM_\Gamma\leq1$ do not contribute. Hence, only the Type~I and Type~II loci contribute:
\begin{equation*}
     0=\res_{t=0} p_*[\overline\cM(W)]^{\vir}= [\overline\cM_{0,k}(X,\beta)]^{\vir}-\pi_*[\overline\cM_{0,k}(\wtil{X},\pi^!\beta)]^{\vir}
\end{equation*}
in $A_{*}(\overline\cM_{0,k}(X,\beta))$.
\end{proof}

Using this master space, we can construct a counterexample to the blow-up formula when $Z\subset X$ is not positive. 
\begin{exam}[Counterexample in the non-positive case]
\label{exam:counterex for non-positive}
 Let $Z=\mathbb P^1$ and $N=\mathcal O_Z(1)\oplus \mathcal O_Z(-3)$ be a rank-two bundle on $Z$. Let $X=\mathbb P_Z(N\oplus 1)$ with $Z=\mathbb P(0\oplus 1)\subset X$. Denote $
\wtil X=\Bl_ZX$ with the exceptional divisor being 
\begin{equation*}
    E=\mathbb P_Z(\mathcal O(1)\oplus\mathcal O(-3))\cong \mathbb F_4.
\end{equation*}

Let $\delta$ be the class of the unique section of self-intersection $-4$, and let $e$ be the fiber class. We now fix $\beta=[Z]$ to be the class of the zero section. The identity
\begin{equation*}
    K_{\wtil X}\cdot\delta=(\pi^*K_X+E)\cdot \delta=E\cdot \delta=1
\end{equation*}
shows $\pi^!\beta=\delta+e$. Recall the master space $W=\operatorname{Bl}_{Z\times0}(X\times \mathbb P^1)$ with its natural $\mathbb C^*$-action.
Note that the fixed line $Z\cong \mathbb P_Z(0\oplus 1)\subset \mathbb P_Z(N\oplus 1)$ has homology class $\pi^!\beta$ in $W$.
By \eqref{eq:vdim1}, \eqref{eq:vdim3}, and Lemma~\ref{lemm:vdim1}, Type~I, II, and III fixed loci have virtual codimension $1$, while there is a unique Type~IV fixed locus. It has a single $0$-vertex of class $\delta$ and a single $*$-vertex which is unstable. They are connected by a single edge of class $\ell$. Using Lemma~\ref{lemm:vdim1}, it has virtual dimension $-1$, so its virtual class vanishes and it makes no contribution. Considering the residue as in Theorem~\ref{thm:main1}, the Type~III locus $\overline{\cM}_{0,0}(Z,[Z])$ over $0\in \mathbb P^1$ has virtual dimension 0 and contributes $-1$ to the residue. Hence, we get
\begin{equation*}
    \deg[\overline{\cM}_{0,0}(\widetilde X,\pi^!\beta)]^{\mathrm{vir}}=\deg[\overline{\cM}_{0,0}(X,\beta)]^{\mathrm{vir}}-1,
\end{equation*}
providing a counterexample when $Z$ is not positive. Similar counterexamples were discovered earlier using different methods. See, for example, \cite[Example~4.16]{Lai09} and \cite[Theorem 1.1]{Ke20}.
\end{exam}

\begin{rmk}
\label{rmk:counterex for g=1}
The virtual cycle identity of Theorem~\ref{thm:main1} does not hold in general in positive genus. For example, consider the case where $X$ is a surface and $Z$ is a point, with $\beta=0$. For $g=1$, we have
\begin{equation*}
    [\overline\cM_{1,1}(\wtil X,0)]^{\vir}= c_2(\mathbb E_{\textrm{Hdg}}^\vee\boxtimes T_{\wtil X})\cap[\overline\cM_{1,1}\times \wtil X]
\end{equation*}
where $\mathbb E_{\textrm{Hdg}}$ is the Hodge bundle on $\overline \cM_{1,1}$. The same holds for $X$. Using $\pi_*(c_1(T_{\wtil X}))=c_1(T_X)$ and $\pi_*(c_2(T_{\wtil X}))=c_2(T_X)+[\pt]$, we obtain
\begin{equation}
\label{eq:counterex for g=1}
    \pi_*[\overline\cM_{1,1}(\wtil X,0)]^{\vir}= [\overline\cM_{1,1}(X,0)]^{\vir}+ [\overline \cM_{1,1}\times\pt].
\end{equation}
Section~\ref{sec:higher genus} introduces conditions for a higer-genuus comparison. For counterexamples with $\beta\neq 0$, see Remark~\ref{rmk:counterex for case ii}.
\end{rmk}
A modification of the remark above provides the following example. It serves as a counterexample to the numerical higher genus blow-up formula in the case of $Z\neq \pt$, $\dim X\leq 3$. See Section~\ref{sec:Introduction} for background.
\begin{rmk}
\label{rmk:counterex for case ii}
Let $S$ be a smooth projective surface and $C$ be a smooth projective genus-one curve. Set $X=S\times C$ and $Z=\pt\times C$, with $\beta=[Z]$. Let $\wtil S$ be the blow-up of $S$ at a point. For $g=1$, we have
\begin{equation*}
    \overline\cM_{1,0}(X,\beta)\cong S,\quad\text{and}\quad [\overline\cM_{1,0}(X,\beta)]^{\vir}=c_2(S)\cap [S] \quad \text{ in } A_0(S).
\end{equation*}
 The same holds for $\wtil X=\wtil {S}\times C$, and we get
\begin{equation*}
    \pi_*[\overline\cM_{1,0}(\wtil X,\pi^!\beta)]^{\vir}= [\overline\cM_{1,0}(X,\beta)]^{\vir}+[\pt]\quad 
\end{equation*}
in $A_0(S)$. This example demonstrates that the assumption that $Z$ is a point is essential for numerical blow-up formulas.
\end{rmk}

\section{Variants of the blow-up formula in genus zero}
Using variants of the master space technique developed in Section~\ref{sec: proof}, we generalize other blow-up formulas  studied in the literature. See Section~\ref{sec:Introduction} for connections with previous work.
\subsection{Blow-up vanishing formula}
In this section, we prove the blow-up vanishing formula. We now take $\beta' = \pi^! \beta + M\ell $ for $\beta \in \rH_2(X;\Z)$, where $M > 0$, and consider the induced $\C^*$-action on the master space $\overline\cM(W) := \overline\cM_{0,k}(W,\beta')$. Recall the description of Type~I to Type~IV loci in Section~\ref{sec: proof}. Note that Type~I and III loci do not appear in this case. Descriptions of Type~II and Type~IV are unchanged except that the total degree is $\pi^!\beta+M\ell$ instead of $\pi^!\beta$. A direct computation shows
\begin{align}\label{eq:vdim2}
\vd \overline\cM_{0,k}(\wtil{X},\pi^!\beta + M\ell)  = \vd \overline\cM(W) - 1 - M.
\end{align}

By an argument parallel to that of Lemma~\ref{lemm:vdim1}, we have the following.
\begin{lem}\label{lemm:vdim2}
The virtual codimension of Type~IV fixed loci is given by:
\begin{equation*}
    \vcd \overline\cM_\Gamma-1-M=\delta_\Gamma +r(m-1)+h+(r-1)(\mathfrak e-h)-2U_1-U_2.
\end{equation*}
In particular, assuming positivity of $Z$, so $\delta_\Gamma>0$, we obtain
\begin{equation*}
    \vcd \overline\cM_\Gamma-1-M> r(m-1)+h-2U_1-U_2.
\end{equation*}
\end{lem}
Moreover, an analogue of Lemma~\ref{lem: contribution vanishes} for $\beta'=\pi^! \beta + M \ell$ holds. The auxiliary moduli $\overline\cM_{\wtil \Gamma}$, $\overline\cM_{\what \Gamma}$, $\overline\cM_{\overline \Gamma}$, and the factorization \eqref{eq: factorization} carry over unchanged to this case. We continue to write $$   \omega_\Gamma:\what\cM_{\Gamma}\longrightarrow\overline\cM_{\wtil \Gamma}$$ for the contraction of free unstable univalent $0$-vertices.

\begin{lem}
\label{lem: contribution vanishes 2}
    If $Z$ is positive, the following vanishing holds:
\begin{equation}
    (\omega_\Gamma)_*\left([t^{-1}]\left(t^M\cdot \frac{[\what\cM_\Gamma]^{\vir}}{e^{\C^*}(\mathfrak b^*N_\Gamma^{\vir})}\right)\right)=0.
\end{equation}
\end{lem}
Virtual localization again gives us
\begin{equation}
\label{eq:localization2}
    [\overline\cM(W)]^{\vir}= \frac{\iota_*[\overline\cM_{0,k}(\wtil{X},\pi^!\beta+M\ell)]^{\vir}}{e^{\C^*}(N^{\vir} |_{\overline\cM_{0,k}(\wtil{X},\pi^!\beta+M\ell)} )}+\sum_{\Gamma \in\Omega}\frac{1}{|\Aut(\Gamma)|}\cdot\frac{\iota_*[\overline\cM_\Gamma]^{\vir}}{e^{\C^*}(N^{\vir} \vert_{\overline\cM_\Gamma})}
\end{equation}

Consider again $p_*[\overline\cM(W)]^{\vir}\in A_{*}(\overline\cM_{0,k}(X,\beta))$ where $p: \overline\cM(W)\to \overline\cM_{0,k}(X,\beta)$ is the morphism induced on moduli. Its $\C^*$-equivariant lift is obtained by applying the equivariant pushforward $p_*$ to \eqref{eq:localization2}. By \eqref{eq:vdim2} and Lemma~\ref{lem: contribution vanishes 2}, all contributions to $\res_{t=0} \left(t^M\cdot p_*[\overline\cM(W)]^{\vir}\right)$ except the Type~II contribution vanish. Hence, we get
\begin{equation*}
     0=\res_{t=0} \left(t^M\cdot p_*[\overline\cM(W)]^{\vir}\right)= (-1)^{M+1}\pi_*[\overline\cM_{0,k}(\wtil{X},\pi^!\beta+M\ell)]^{\vir}
\end{equation*}

This proves the following vanishing theorem.

\begin{thm}\label{thm:main2}
Assuming positivity of the blow-up center $Z$ and $M>0$, we have
\begin{equation*}
\pi_*[\overline\cM_{0,k}(\wtil X,\pi^!\beta+M\ell)]^{\vir}=0
\end{equation*}
in $A_*(\overline\cM_{0,k}(X,\beta))$. In particular, we have 
\begin{equation*}
\langle \pi^*\gamma_1,\cdots, \pi^*\gamma_k\rangle^{\wtil{X}}_{0,\pi^!\beta + M\ell}=0
\end{equation*}
for $\gamma_i\in \rH^*(X;\Z)$.  
\end{thm}

\begin{rmk}
Recently, Chen, Du, and Fan \cite{CDF26} proved Theorem~\ref{thm:main2} using orbifold Gromov--Witten theory and root stacks. In their approach, the balancing condition along the root divisor excludes univalent unstable $0$-vertices. By contrast, our approach remains within ordinary Gromov--Witten theory and handles loci of insufficient virtual codimension through the geometric contraction of free leaves and a pushforward-vanishing argument.
\end{rmk}

The following example gives a counterexample to Theorem~\ref{thm:main2} without positivity, by a construction analogous to Example~\ref{exam:counterex for non-positive}.
\begin{exam}[Counterexample to the blow-up vanishing formula]
\label{ex:counterex-vanishing}
Let $Z=\mathbb P^1$ and let $N=\mathcal O_Z(-1)\oplus\mathcal O_Z(-2)$. Consider $X=\mathbb P_Z(N\oplus 1)$ with $   Z=\mathbb P_Z(0\oplus 1)\subset X$
and set $\widetilde X=\Bl_ZX$. Note that $Z\subset X$ is not positive. The exceptional divisor is
\begin{equation*}
    E=\mathbb P_Z(N)\cong\mathbb F_1.
\end{equation*}
Let $\delta$ be the class of the unique section of self-intersection $-1$, and let $e$ be the fiber class
of $E\to Z$. We fix $\beta=[Z]$, the class of the zero section. As in Example~\ref{exam:counterex for non-positive}, one can show that
\begin{equation*}
    \pi^!\beta+e=\delta.
\end{equation*}
 Recall the master space  $W=\Bl_{Z\times\{0\}}(X\times\mathbb P^1)$ with its natural $\mathbb C^*$-action, and consider
$\overline{\mathcal M}_{0,0}(W,\delta)$, where
$\delta=\pi^!\beta+e$. There are two contributing fixed loci: first is the Type~II locus $\overline{\mathcal M}_{0,0}(\widetilde X,\delta)$. Its contribution to the residue
\begin{equation*}
    \res_{t=0}\left(
    t\cdot p_*[\overline{\mathcal M}_{0,0}(W,\delta)]^{\vir}\right) 
\end{equation*}
is $\pi_*[\overline{\mathcal M}_{0,0}(\widetilde X,\delta)]^{\vir}$. The other is the unique Type~IV locus. Its graph has a single
$\ast$-vertex of degree $\beta=[Z_\ast]$ and a single degree-one
invariant edge ending at an unstable univalent $0$-vertex. Denote this
graph by $\Gamma_0$. Its fixed locus is
\begin{equation*}
    \overline{\mathcal M}_{\Gamma_0}
    \cong
    \overline{\mathcal M}_{0,1}(Z,\beta)\times_Z E
    \cong E.
\end{equation*}
By Lemma~\ref{lem:fixed locus gysin}, $[\overline{\mathcal M}_{\Gamma_0}]^{\vir}=[E]$. Let $\phi:E\to Z$ be the projection. Denote
\begin{equation*}
    H=\phi^*\left(c_1(\mathcal O_Z(1))\right)\quad\text{and}\quad
    \xi=c_1(\cO_E(1)).
\end{equation*}
Normalizing the source at its unique node and cancelling the moving
edge deformations against the matching conditions gives
\begin{equation*}
    [N_{\Gamma_0}^{\vir}]
    =
    [(\phi^*T_Z\otimes \cO_E(-1))\otimes\mathbf t]
    -[H^1(Z,\mathcal O_Z(-2))\otimes\mathbf t]
\end{equation*}
in equivariant $K$-theory. Hence, 
\begin{equation*}
    \frac{1}{e_{\mathbb C^*}(N_{\Gamma_0}^{\vir})}
    =\frac{t}{t+2H-\xi}.
\end{equation*}
Since $\int_E(2H-\xi)^2=-1$, after pushforward to
$\overline{\mathcal M}_{0,0}(X,\beta)\cong\mathrm{pt}$, the Type~IV
contribution is
\begin{equation*}
    \res_{t=0}\left( t\cdot (p_{\Gamma_0})_*
    \left(
        \frac{[\overline{\mathcal M}_{\Gamma_0}]^{\vir}}
        {e_{\mathbb C^*}(N_{\Gamma_0}^{\vir})}
    \right)\right) = \left(\int_E(2H-\xi)^2\right)[\mathrm{pt}] 
    =-[\mathrm{pt}].
\end{equation*}
Hence, taking the residue as in the proof of
Theorem~\ref{thm:main2}, we obtain
\begin{equation*}
    0=\res_{t=0}\left(t\cdot p_*[
        \overline{\mathcal M}_{0,0}(W,\delta)
    ]^{\vir}\right)=\pi_*[\overline{\mathcal M}_{0,0}(\widetilde X,\delta)]^{\vir}-[\mathrm{pt}].
\end{equation*}
Thus $\pi_*[\overline{\mathcal M}_{0,0}(\widetilde X,\pi^!\beta+e)
    ]^{\vir}=[\mathrm{pt}]\neq 0$. This provides a projective counterexample when the positivity assumption is omitted.
\end{exam}

\subsection{Relative blow-up formula}\label{sec: rel blowup formula}
Using the same idea, we prove a relative version of the blow-up formulas, again under positivity assumptions on the normal bundle. We recall the notation for expanded targets and relative stable maps in \cite{Li01} before stating the relative blow-up formula. Every connected component of the source has genus zero in this section.

Let $W\to B$ be a flat family over a smooth curve $B$ with a chosen local parameter $B\to \AA^1$ where $W_t\coloneqq W\times_B\{t\}$ is smooth for $t\in B\setminus\{0\}$. Assume that $W$ is a simple degeneration and $W_0$ splits into two smooth irreducible components $W_0=Y_1\cup_D Y_2$, meeting along a divisor $D$. This means that $(Y_1,D)$ and $(Y_2,D)$ are relative pairs where $D$ is a smooth divisor in $Y_i$. Let $\mathbf d=(g,k,b)$ be the genus, number of markings, and degree data. Then, we denote by $\fM(\fW,\mathbf d)$ the moduli of stable maps to the stack of expanded degenerations $\fW$ over $(\Sch/B)$. The stack $\fW$ parametrizes all expanded degenerations $W[l]\to B[l]\coloneqq B\times_{\AA^1}{\AA^{l+1}}$ for $l\in \Z_{\geq 0}$ whose special fiber $W[l]_0$ can be described as
\begin{equation*}
    W[l]_0= Y_1 \cup_D P\cup_D\cdots \cup_DP\cup_D Y_2\quad (l\text{ copies of } P)
\end{equation*}
where $P\coloneqq \PP_D(N_{D/Y_1}\oplus1)$. The infinity section of the $i$th copy of $P$ is glued to the zero section of the $(i+1)$st copy. Moreover, let $H_i\subset \AA^{l+1}$ be the $i$-th coordinate hyperplane and $B[l]_i\coloneqq B\times_{\AA^1}H_i$ so that $B[l]_0=\bigcup_{i=1}^{l+1} B[l]_i$. Then, $W[l]\times_{B[l]}B[l]_i$ is a union of two irreducible components $\Delta_+$ and $\Delta_-$ glued along a nodal divisor $\bD_i$. The moduli $\overline\cM(\fW,\mathbf d)$ parametrizes families of pre-deformable stable morphisms to $\fW$. Pre-deformability ensures that for each node $p\in C$ lying in the preimage of $\bD_i$, the two branches of $C$ at $p$ map to different irreducible components $\Delta_\pm$, with identical contact orders. We also let $\overline\cM(\fW_0,\mathbf d)=\overline\cM(\fW,\mathbf d)\times_B0$ be the moduli of stable maps to the singular scheme $W_0$.

Similarly, \cite{Li01} defines the moduli $\overline\cM_\Psi(Y,D)$ of relative stable maps to a relative pair $(Y,D)$. We do not assume the source curve to be connected. In the relative setting, $\Psi$ consists of the number $k$ of absolute markings, the degree $\beta$, and $q$ relative markings yielding a partition of length $q$ of $\beta\cdot D$. One can encode these data in a graph $\Psi$ with $\Ver (\Psi)$ indexing the connected components of the source curve. The vertices have legs corresponding to absolute markings and roots (with integer-valued weight function $\mu$ on roots) corresponding to relative markings. We also have a genus function $g:\Ver(\Psi)\to \Z_{\geq 0}$, and a weight function $b:\Ver(\Psi)\to \rH_2(Y)$ corresponding to the degree. The graph $\Psi$ is said to be \textit{admissible} if $\Ver(\Psi)$ consists of a single vertex, or every vertex carries at least one root.

We also allow the target to expand and consider all (relative) expanded degenerations $(Y[l], D_\infty)\to \AA^l$ whose special fiber $Y[l]_0$ can be described as
\begin{equation*}
    Y[l]_0= Y \cup_D P\cup_D\cdots \cup_DP \quad (l\text{ copies of } P)
\end{equation*}
so that the singular locus is $\Sing Y[l]_0=\bigcup_{i=0}^{l-1}D_i$. We will also use the notation 
\begin{equation*}
    D[l]= P\cup_D\cdots \cup_DP \quad (l\text{ copies of } P).
\end{equation*}
Under this description, $D_\infty\vert_0$ is the infinity divisor of the last $P$ component. Moreover, $Y[l]\times_{\AA^l}H_i$ is a union of two irreducible components glued along a nodal divisor $\bB_i$. Again, pre-deformability means that for each node $p\in C$ lying in the preimage of $\bB_i$, the two branches of $C$ at $p$ map to different irreducible components with identical contact orders. The moduli $\overline\cM_\Psi(Y,D)$ parametrizes families of pre-deformable stable morphisms to all expanded degenerations $Y[l]$.
We now recall the description of the relative deformation--obstruction theory of $\overline\cM_\Psi(Y,D)$ relative to curve moduli $\mathcal T\times \mathfrak M_{\Psi}$ from \cite{GV05}. Here $\cT$ is an open substack of $\mathfrak M_{0,3}$ where nodes separate the $\infty$-marking from markings $0$ and $1$, and $\mathfrak M_{\Psi}$ is the corresponding product of prestable curve stacks. Let $f^\dagger T_{Y[l]}(-\log D_\infty)$ be the quotient of $f^*T_{Y[l]}(-\log D_\infty)$ by its torsion subsheaf. The relative deformation space of $f$ is given by 
\begin{equation*}
    \mathrm{RelDef}(f)=H^0(C,f^\dagger T_{Y[l]}(-\log D_\infty)).
\end{equation*}
The obstruction space of $f$ satisfies the following short exact sequence:
\begin{equation}
\label{eq:Relobstruction}
0\to H^1(C,f^\dagger T_{Y[l]}(-\log D_\infty))\to    \mathrm{RelOb}(f)\to \oplus_{i=0}^{l-1} L_i^{\oplus j_i} \to 0
\end{equation}
where $L_i$ is the target-node smoothing line of $D_i$ in $Y[l]$ and $j_i$ is the number of nodes of $C$ mapping to $D_i$. The $L_i^{\oplus j_i}$ summands in this exact sequence correspond to the pre-deformability obstructions to deforming $f$. Geometrically, sections of $f^\dagger T_Y[l](-\log D_\infty)$ are vector fields of $T_Y[l]$ along $C$ whose values at nodes lie in the corresponding tangent spaces of $D_i$. The following sequence can be verified by a local computation:
\begin{equation}
\label{eq:Dagger SES}
    0\to f^\dagger T_{W_0}\to f_1^*T_{Y_1}(-\operatorname{log} D)\oplus f_2^*T_{Y_2}(-\operatorname{log} D)\to \bigoplus_{x\mapsto D}\ev_x^*T_D\to 0 
\end{equation}
where $C=C_1\cup C_2$ and $f_i=f\vert_{C_i}$. The maps are induced by natural restrictions.\\

Our goal in this section is to prove a relative version of the blow-up conjecture, assuming positivity of $Z$. Let $(X,D)$ be a relative pair such that $D$ does not intersect $Z$. Following the notation of \cite{Lai09}, define $b(\Psi)=\sum_{w\in\Ver(\Psi)}b(w)$ to be the total degree for an admissible graph $\Psi$. Fix an admissible graph $\Psi$ associated to $(X,D)$. Then, there is a canonical admissible graph $\pi^!(\Psi)$ associated to $(\wtil X,D)$ characterized by:
\begin{itemize}
    \item [(1)] All data of $\pi^!(\Psi)$ agree with those of $\Psi$, except for the weight function.
    \item [(2)] The weight function $b$ is determined by the following commutative diagram.
    \begin{equation*}
        \begin{tikzcd}
	{\Ver(\Psi)} & {\rH_2(X)} \\
	{\Ver(\pi^!(\Psi))} & {\rH_2(\wtil X)}
	\arrow["b_\Psi", from=1-1, to=1-2]
	\arrow[from=1-1, to=2-1]
	\arrow["{\pi^!}", from=1-2, to=2-2]
	\arrow["b_{\pi^!(\Psi)}"', from=2-1, to=2-2]
\end{tikzcd}
    \end{equation*}
\end{itemize}
Finally, we fix an admissible graph $\wtil\Psi$ associated to $(\wtil X,D)$ whose data are identical to those of $\Psi$ except for the weight function. Assume it satisfies
\begin{equation*}
    b(\wtil \Psi)=\pi^!b( \Psi).
\end{equation*}

Recall the construction $W= \Bl_{Z \times \set{0}} (X \times \PP^1)$ with a 
distinguished divisor $D\times \PP^1\subset W$. Let $\Psi'$ be the admissible graph of $(W,D\times\PP^1)$ whose data are identical to those of $\wtil \Psi$, except for the weight function characterized by the commutative diagram
\begin{equation*}
    \begin{tikzcd}
	{\Ver(\wtil \Psi)} & {\rH_2(\wtil X)} \\
	{\Ver(\Psi')} & {\rH_2(W)}
	\arrow["b_{\wtil \Psi}", from=1-1, to=1-2]
	\arrow[from=1-1, to=2-1]
	\arrow["{{\iota_*}}", from=1-2, to=2-2]
	\arrow["b_{\Psi'}"', from=2-1, to=2-2]
\end{tikzcd}
\end{equation*}
where $\iota:\wtil X\into W$ is the inclusion of $\wtil X$ as an irreducible component of $W\times_{\PP^1}\{0\}$.
This time, we use the master space
\begin{equation*}
\overline\cM(W) := \overline\cM_{\Psi'}(W,D\times\PP^1).
\end{equation*}
There is a natural $\C^*$-action on $\overline\cM(W)$ induced by the action on $\PP^1$. The fixed loci are analogous to Types~I to IV in Section~\ref{sec: proof}, with the following modifications:
\begin{itemize}
\item[I.]
Replace $\overline\cM_{0,k}(X, \beta )$ with $\overline\cM_\Psi(X, D)$ if $\wtil \Psi=\pi^!(\Psi)$. If $\wtil \Psi\neq\pi^!(\Psi)$, the Type~I locus is empty as in the proof of Theorem~\ref{thm:main2}.

\item[II.]
Replace $\overline\cM_{0,k}(\wtil{X},\pi^! \beta)$ with $\overline\cM_{\wtil\Psi}(\wtil{X},D)$.

\item[III.]
If the weight function $\mu$ on roots is nonzero, the Type~III locus is empty.

\item[IV.] The description for a graph $\Gamma$ is identical, but the moduli factors at $0$-vertices are replaced by moduli spaces of relative stable maps to $(\wtil X,D)$. Denote the moduli by $\overline\cM_\Gamma^{\rel}$, omitting $\wtil \Psi$ from the notation.
\end{itemize}
Using the same strategy as in Theorem~\ref{thm:main1}, we have the following genus zero relative blow-up formula. Assume that the target moduli space $\overline\cM_{\Psi}(X,D)$ is nonempty. Let $$p: \overline\cM(W)\to \overline\cM_{\Psi}(X,D)$$
be the morphism on relative moduli induced from $(W,D\times\PP^1)\to(X,D)$.
\begin{thm}
\label{thm: relative blow-up formula}
    Let $(X,D)$ be a relative pair such that $D\cap Z=\emptyset$ ($D$ is away from the blow-up locus) and $Z$ is positive. Let $\widetilde \Psi$ and $\Psi$ be admissible graphs of $(\widetilde X,D)$ and $(X,D)$, respectively, identical except for the weight function $b$. Assume $b(\widetilde \Psi)=\pi^!b(\Psi)$ so that $\vd \overline\cM_{\widetilde \Psi}(\widetilde X,D)=\vd \overline\cM_{\Psi}(X,D)$. We also assume $\pi_*\circ b_{\wtil\Psi}=b_\Psi$ so that 
    \begin{equation*}
        \pi:\overline\cM_{\widetilde \Psi}(\widetilde X,D)\longrightarrow\overline\cM_{\Psi}(X,D)
    \end{equation*}
    is well-defined. Then, we have the following identity of virtual cycles for genus zero relative stable maps. 
    \begin{align*}
        \pi_*[\overline\cM_{\widetilde \Psi}(\widetilde X,D)]^{\vir}=
        \begin{cases}
            [\overline\cM_{\Psi}(X,D)]^{\vir} &\text{if} \quad \widetilde \Psi=\pi^!(\Psi)\\\\
            0& \text{if}\quad \widetilde \Psi\neq\pi^!(\Psi)
        \end{cases}
    \end{align*}
    in $A_{*}(\cM_{\Psi}(X,D))$. Consequently, the relative invariants satisfy
    \begin{align*}
        \langle \pi^*\gamma_1,\pi^*\gamma_2,\cdots, \pi^*\gamma_k\rangle_{\widetilde \Psi}^{\widetilde X,D}=
        \begin{cases}
            \langle\gamma_1,\gamma_2,\cdots,\gamma_k\rangle_{\Psi}^{X,D} &\text{if} \quad \widetilde \Psi=\pi^!(\Psi)\\\\
            0& \text{if}\quad \widetilde \Psi\neq\pi^!(\Psi)
        \end{cases}
    \end{align*}
    in $\rH_*(D^q)$ where $q$ is the number of roots.
\end{thm}
\begin{proof}
The obstruction theory of $\overline\cM_\Psi(X,D)$ constructed in \cite{Li02} has the following form. Its global map-deformation term
is obtained by replacing $f^*T_X$ with the sheaf
$f^\dagger T_{X[l]}(-\log D_\infty)$; the full relative theory also
contains the pre-deformability terms in~\eqref{eq:Relobstruction} and the
deformation--automorphism data of the expanded target. Let $I$ be the
common vertex set of $\Psi$ and $\wtil\Psi$. Since
$b(\Psi)\cdot D=b(\wtil \Psi)\cdot D=\sum_{\text{roots}}\mu$, and the
number $q$ of relative markings is the same for $\Psi$ and
$\wtil\Psi$, we have
\begin{equation*}
    \vd \overline\cM_{\wtil \Psi}(\wtil X,D)
    =|I|\cdot(\dim X-3)
    +\int_{b(\Psi)}(c_1(X)-D)+k+q.
\end{equation*}
In particular, the Type~II locus has virtual codimension $|I|$. If
$|I|=1$, the proof of Theorem~\ref{thm:main1} applies verbatim, since
the relative divisor is disjoint from the blow-up locus. It therefore remains to treat $|I|>1$.

(1) First, assume $\wtil\Psi=\pi^!(\Psi)$. Let
\begin{equation*}
    q^{\exp}:W^{\exp}\longrightarrow\PP^1
\end{equation*}
be the projection from the expanded universal target. For a labeled
connected source component $w\in I$, we denote the restrictions of the
universal curve and universal map by
\begin{equation*}
    \rho_w:\cC_w\longrightarrow\overline\cM(W),
    \qquad
    f_w:\cC_w\longrightarrow W^{\exp}.
\end{equation*}
Define
\begin{equation*}
    K_{w,x}=R\rho_{w*}f_w^*(q^{\exp})^*\cO(x)
    \qquad\text{for }x\in\{0,\infty\}\subset\PP^1,
\end{equation*}
and let $\Theta_{w,x}=e^{\C^*}(K_{w,x})$ be its equivariant Euler class
on $\overline\cM(W)$. For a subset $I'\subset I$ and $w\notin I'$,
consider the codimension-$(|I|-1)$ class
\begin{equation*}
    A_{I',w}
    =\prod_{i\in I'}\Theta_{i,0}
     \cdot
     \prod_{i\notin I'\cup\{w\}}\Theta_{i,\infty}
\end{equation*}
in the operational Chow group of $\overline\cM(W)$. Since
Lemma~\ref{lem: contribution vanishes} can be applied individually to
each connected component, 
we can use Type~IV vanishing as in Theorem~\ref{thm:main1}. The componentwise argument takes place on the full relative fixed locus; it does not use a product decomposition by connected source components. For every relative fixed graph
$\Gamma$, pass to to a single finite cover
\begin{equation*}
    \mathfrak b_\Gamma:\what\cM_\Gamma\longrightarrow
    \overline\cM_\Gamma^{\mathrm{rel}}
\end{equation*}
on which all free unstable univalent $0$-vertices are labeled. For $w\in I$, let $U_{1,w}^0$ be the number of such leaves. Contract these leaves and retain their attaching brances as markings $a_\mathfrak l$. The common expanded target and
all other source components are retained. This gives a morphism
\begin{equation*}
    \omega_{\Gamma,w}:
    \what\cM_\Gamma\longrightarrow\mathcal B_{\Gamma,w}.
\end{equation*}
It is an iterated $\mu_{d_e}$-gerbe over projective bundles, with
successive projective-bundle factors
$\PP(\ev_{a_\mathfrak l}^*N_{Z/X})$ for each leaf $\mathfrak l$, and hence has relative
dimension $(r-1)U_{1,w}^0$, as in the absolute case. In particular, we get
\begin{equation*}
    [\what\cM_\Gamma]^{\vir}
    =\omega_{\Gamma,w}^!
      [\mathcal B_{\Gamma,w}]^{\vir}.
\end{equation*}
This compatibility includes the global logarithmic deformation terms;
since $D\cap Z=\varnothing$, the pre-deformability terms
$\bigoplus_iL_i^{\oplus j_i}$ in~\eqref{eq:Relobstruction} and the
deformation--automorphism data of the common expansion are unchanged under pullback.
On the same cover, there are moving virtual bundles
$\mathsf K_{\Gamma,w}$ such that
\begin{equation*}
    [\mathfrak b_\Gamma^*N_\Gamma^{\vir}]
    =\sum_{w\in I}[\mathsf K_{\Gamma,w}]
    \qquad\text{in }
    K^0_{\C^*}(\what\cM_\Gamma).
\end{equation*}
Under $\omega_{\Gamma,w}$, the bundles
$\mathsf K_{\Gamma,v}$ for $v\neq w$, as well as all localization
factors not belonging to $C_w$, pullback from
$\mathcal B_{\Gamma,w}$. Thus the virtual-pushforward arguments of
Lemmas~\ref{lem: contribution vanishes} and
\ref{lem: contribution vanishes 2} apply by the projection formula. Hence, taking the localization residue of $p_*(A_{I',w}\cap[\overline\cM(W)]^{\vir})$ gives
\begin{equation*}
    p_*[F_{I'}]^{\vir}
    =p_*[F_{I'\cup\{w\}}]^{\vir}
\end{equation*}
in $A_*(\overline\cM_\Psi(X,D))$. Here $F_{I'}$ is the locus in which
components $i\in I'$ are of Type~II and components $i\notin I'$ are of
Type~I. Iterating these equalities proves the theorem in this case.

(2) Second, assume $\wtil\Psi\neq\pi^!(\Psi)$. Let $I_0\subsetneq I$ be
the set of vertices satisfying
$b_{\wtil\Psi}(w)=\pi^!b_\Psi(w)$. Consider the degree-$|I_0|$ class
\begin{equation*}
    A_0=\prod_{i\in I_0}\Theta_{i,0}
\end{equation*}
in the operational Chow group of $\overline\cM(W)$. We claim
that the $t^{-|I|+|I_0|}$-coefficient of
$p_*(A_0\cap[\overline\cM(W)]^{\vir})$ receives contributions only from
Type~II loci. For $w\in I\setminus I_0$, write
\begin{equation*}
    b_{\wtil\Psi}(w)=\pi^!b_\Psi(w)+M_w\ell,
\end{equation*}
and set $M_w=0$ for $w\in I_0$. Then $\sum_{w\in I}M_w=0$. On each fixed locus, set
\begin{equation*}
    a_w=\operatorname{rank}\mathsf K_{\Gamma,w},
    \qquad
    \frac{1}{e^{\C^*}(\mathsf K_{\Gamma,w})}
    =
    t^{-a_w}\sum_{\nu\geq0}
    P_{\Gamma,w,\nu}\,t^{-\nu}.
\end{equation*}
For each $w$, let $\nu_w$ denote the exponent chosen from this expansion. On every fixed locus not annihilated by $A_0$, each $C_w$ is of
Type~II or Type~IV, and $A_0$ restricts to a nonzero scalar multiple of
$t^{|I_0|}$. The dimensional arguments of
Lemmas~\ref{lem: contribution vanishes} and
\ref{lem: contribution vanishes 2} apply uniformly in the sign of
$M_w$ and show that, whenever $C_w$ is of Type~IV,
\begin{equation*}
    (\omega_{\Gamma,w})_*
    \left(
        P_{\Gamma,w,\nu_w}
        \cap[\what\cM_\Gamma]^{\vir}
    \right)
    =0
    \qquad\text{for }\nu_w\leq 1+M_w-a_w.
\end{equation*}
The coefficient condition, together with
$\sum_{w\in I}M_w=0$, gives
\begin{equation*}
    \sum_{w\in I}\nu_w
    =
    \sum_{w\in I}(1+M_w-a_w).
\end{equation*}
For every Type~II component $C_w$, we have 
$\nu_w\geq0=1+M_w-a_w$. For a Type~IV component to contribute, it
must satisfy $\nu_w>1+M_w-a_w$. This contradicts the preceding equality. Therefore only Type~II loci
contribute, proving the claim and hence the desired vanishing.
\end{proof}

\subsection{Generalization of Gathmann's formula}
In this section, we prove a blow-up formula in genus 0 generalizing \cite[Corollary 3.2]{Gat01}. Recall the localization graph description of Type~IV loci from Section~\ref{sec: proof}. Let $\Gamma_{\Gat}$ be the localization graph with a $0$-vertex $v$ and a $*$-vertex $u$. Let there be $k$ legs attached to $v$ and one leg attached to $u$. The degrees are given by $\beta_v=\pi^!\beta-\ell$, $\beta_u=0$, and $\beta_e=\ell$. As a corollary of Lemma~\ref{lem:fixed locus gysin}, we have an embedding 
\begin{equation*}
    \iota_{\Gat}:\overline\cM_{\Gamma_{\Gat}}\into \overline\cM_{0,k+1}(\wtil X,\pi^!\beta-\ell)
\end{equation*}
whose image is $\ev_{k+1}^{-1}(E)$. Lemma~\ref{lem:fixed locus gysin} also identifies the corresponding virtual cycle.

\begin{coro}
\label{cor:M_Gat}
    We have
    \begin{equation*}
    \iota_{\Gat*}[\overline\cM_{\Gamma_{\Gat}}]^{\vir}=\ev_{k+1}^*[E]\cap [\overline\cM_{0,k+1}(\wtil X,\pi^!\beta-\ell)]^{\vir}
    \end{equation*}
    in $A_*(\overline\cM_{0,k+1}(\wtil X,\pi^!\beta-\ell))$.
\end{coro}

 For this subsection, use the master space $\overline\cM(W) := \overline\cM_{0,k+1}(W,\pi^!\beta)$. The description above will yield $[\overline\cM_{\Gamma_{\Gat}}]^{\vir}$. Let $\cZ\subset W$ be the proper transform of $Z\times \PP^1\subset X\times \PP^1$. The $\C^*$-equivariant class of $\cZ$ has the following restrictions.
 \begin{equation}
     \label{eq: restrictions of Z-insertion}
     [\cZ]_{\C^*}\vert_{X_\infty}=[Z], \quad[\cZ]_{\C^*}\vert_{\wtil X}=0,\quad[\cZ]_{\C^*}\vert_{Z_*}=\prod_{i=1}^r(t+\delta_i)
 \end{equation}
 where $\delta_i$ are the Chern roots of $N_{Z/X}$. We recall the following notation for a Type~IV graph $\Gamma$ used in Section~\ref{sec: Type IV vanishing}.
 \begin{equation}
 \label{eq: summary of notation}
 \begin{gathered}
     n=|V_0|,\quad m=|V_*|,\quad h=|\Edg|,\quad \delta_\Gamma=c_1(N_{Z/X})\cdot \beta_*+r, \\ a_\Gamma=\vcd\overline\cM_\Gamma,\quad
     U_i=|V_i^{\un}(\Gamma)|,\quad S=n+m-U_1-U_2,\\
     U_1^\bullet=\text{number of univalent unstable $\bullet$-vertices},\quad (\bullet=0,*)\\
     \omega_\Gamma:\what\cM_\Gamma\longrightarrow\overline\cM_{\wtil \Gamma}
 \end{gathered}
 \end{equation}
 $\omega_\Gamma$ is the contraction of $\overline\cM_e$ attached to unstable univalent $0$-vertices.
\begin{thm}\label{thm:Gathmann}
Assuming positivity of the blow-up center $Z$ and $\beta\neq0$, we have the following generalization of Gathmann's blow-up formula
\begin{equation*}
\langle \pi^*\gamma_1,\cdots, \pi^*\gamma_k\rangle^{\wtil{X}}_{0,\pi^!\beta-\ell}=\langle \gamma_1,\cdots, \gamma_k,[Z]\rangle^X_{0,\beta}    
\end{equation*}
for $\gamma_i\in \rH^*(X;\Z)$.   
\end{thm}
\begin{proof}
We may assume $\gamma_i\neq1$ for all $i$. Consider the insertion class $$A = \prod_{i=1}^k \ev_i^*\gamma_i$$ whose degree equals $\vd\overline\cM_{0,k}(X,\beta)-r+1$. We suppress pullbacks of insertion classes when the relevant maps are clear from the context. Consider the residue of 
\begin{equation}
\label{eq: Gathmann residue}
\mathrm{str}_*\left([\overline\cM(W)]^{\vir}\cap A\cap \ev_{k+1}^*[\cZ]\right)    
\end{equation}
where $\mathrm{str}$ is the structure morphism $\overline\cM(W)\to\pt$.
By the second restriction in \eqref{eq: restrictions of Z-insertion}, the $(k+1)$-st marking of Type~IV graphs should lie in $Z_*$. In particular, the Type~II locus does not contribute to the residue. This time, because the insertion $[\cZ]$ contributes powers of $t$ of degree at most $t$, Type~IV graphs $\Gamma$ with $a_\Gamma>r+1$ will not contribute to the residue. Hence, we assume $a_\Gamma\leq r+1$. Recall the virtual dimension calculation of Lemma~\ref{lemm:vdim1}. We have
\begin{equation}
\label{eq: vdim for Gathmann}
    r+1-a_\Gamma\leq U_1^0\leq(r-1)\cdot U_1^0,
\end{equation}
where the first inequality follows from
\begin{equation*}
    a_\Gamma-1-r+U_1^0=(\delta_\Gamma-r)+r(m-U_1^*-1)+(r-1)U_1^*+(S-1)+(r-1)(\mathfrak e-h)\geq 0.
\end{equation*}
Here $m-U_1^*\geq1$ follows from the fact that the $(k+1)$-st marking lies in $Z_*$, and $\delta_\Gamma\geq r$ (see Remark~\ref{rmk: alternative of positivity}). If the first inequality of \eqref{eq: vdim for Gathmann} is strict, $[\what \cM_{\Gamma}]^{\vir}=\omega_\Gamma^![\overline\cM_{\wtil \Gamma}]^{\vir}$ and virtual pushforward formulas prove the vanishing of contributions, as in Lemma~\ref{lem: contribution vanishes}. Indeed, if we write the expansion in $t^{-1}$ as
\begin{equation*}
    \frac{1}{e^{\C^*}\left(N^{\vir}_{\overline\cM_\Gamma}\right)}=t^{-a_\Gamma}\sum_{\nu\geq 0}P_{\Gamma,\nu}t^{-\nu},
\end{equation*}
then from \eqref{eq: restrictions of Z-insertion}, we have
\begin{equation*}
    [t^{-1}]\frac{\ev_{k+1}^*[\cZ]}{e^{\C^*}\left(N^{\vir}_{\overline\cM_\Gamma}\right)}=\sum_{j+\nu=r+1-a_\Gamma}\ev_{k+1}^*c_j(N_{Z/X})\cdot P_{\Gamma,\nu}.
\end{equation*}
Both $A$ and the classes $c_j(N_{Z/X})$ descend through $\omega_\Gamma$, hence $\nu\leq r+1-a_\Gamma$ yields vanishing for dimensional reasons. Equality in the first inequality of \eqref{eq: vdim for Gathmann} can hold only when
\begin{equation*}
    \delta_\Gamma=r,\quad m=1,\quad U_1^*=0,\quad S=1,\quad \mathfrak e=h.
\end{equation*}
The relation $\sum_{v\in V_0}e_v=e>0$ forces at least one $0$-vertex to be stable. Since $S=1$, the unique $*$-vertex is therefore unstable. Because $U_1^*=0$ and this vertex carries the $(k+1)$-st leg, it is bivalent and incident to exactly one edge. Hence, we have $n=m=h=\mathfrak e=1$, which singles out $\Gamma=\Gamma_{\Gat}$. Its virtual cycle was computed in Corollary~\ref{cor:M_Gat}. It remains to show that Type~III loci do not contribute. For an effective class $\beta_Z\in \rH_2(Z;\Z)$ satisfying $\iota_*\beta_Z=\beta$, the positivity assumption gives $c_1(N_{Z/X})\cdot \beta_Z\geq0$. If $c_1(N_{Z/X})\cdot \beta_Z>0$, then the virtual codimension of Type~III fixed locus is greater than $r+1$, hence does not contribute to the residue. If $c_1(N_{Z/X})\cdot \beta_Z=0$, its contribution is
\begin{equation*}
    -\langle\gamma_1\vert_Z,\cdots, \gamma_k\vert_Z,1\rangle_{0,\beta_Z}^Z=0
\end{equation*}
by the string equation. Here we use the assumption $\beta\neq 0$. By the discussion above, the residue formula for \eqref{eq: Gathmann residue} simplifies to the following.
\begin{multline*}
\res_{t=0} \left( \int\frac{[\overline\cM_{0,k+1}(X,\beta)]^{\vir} \cap A\cap\ev_{k+1}^*[Z]}{e^{\C^*}\left(N_{\overline\cM_{0,k+1}(X,\beta)}^{\vir}\right)}\right.\\
\left.+\int\frac{\left([\overline\cM_{0,k+1}(\wtil X,\pi^!\beta-\ell)]^{\vir} \cap A\cap \ev_{k+1}^*[E]\right)\cdot\prod_{i=1}^r (t+\delta_i)}{e^{\C^*}\left(N_{\overline\cM_{\Gamma_\Gat}}^{\vir}\right)} \right) = 0.
\end{multline*}
Again, we have
\begin{equation*}
\frac{1}{e^{\C^*}\left(N_{\overline\cM_{0,k+1}(X,\beta)}^{\vir}\right)} = \frac{1}{t}\quad\text{and}\quad \frac{1}{e^{\C^*}\left(N_{\overline\cM_{\Gamma_\Gat}}^{\vir}\right)} = \frac{-1}{t^{r+1}} + O(t^{-r-2}).
\end{equation*}
For the second expansion, deformations of $C_e$ contribute $t^r$ to the moving part of the normal bundle and smoothing the node at the $\wtil X$-endpoint contributes $-t$. Substituting these into the residue formula gives
\begin{align*}
    \int_{[\overline\cM_{0,k+1}(X,\beta)]^{\vir}}A \cap \ev_{k+1}^*[Z]&= \int_{[\overline\cM_{0,k+1}(\wtil X,\pi^!\beta-\ell)]^{\vir}}A\cap \ev^*_{k+1}[E]\\
    &=\int_{[\overline\cM_{0,k}(\wtil X,\pi^!\beta-\ell)]^{\vir}} A
\end{align*}
by the divisor equation.
\end{proof}

\section{Higher genus results via degeneration}
The preceding sections used the master space method to establish blow-up formulas in genus zero. In this section, we focus on how equivariant degeneration methods can be used to prove results in higher genus. Section~\ref{sec:higher genus} proves an absolute--relative correspondence for genus $g\leq 1$ (Theorem~\ref{thm: main Abs-Rel corr}). For $g\geq 2$, an additional assumption on the dimension of $Z$ is required (Theorem~\ref{thm: all genus Abs-Rel corr}). Although this assumption may appear technical, it cannot in general be omitted, as demonstrated in Section~\ref{sec:counterex of Hu's conj} by a counterexample to Conjecture \ref{conj:Hu}.

\subsection{Absolute--relative correspondence}
\label{sec:higher genus}
Let $X$ be a smooth projective variety of dimension $d$, let $Z$ be a smooth subvariety of codimension $r$, and fix a curve class $\beta$. Fix a genus $g\geq 0$ and abbreviate the moduli space by $\overline\cM(X)\coloneqq\overline\cM_{g,k}(X,\beta)$. Denote by $E$ the exceptional divisor of $\wtil X=\Bl_ZX$ and consider the relative pair $(\wtil X,E)$ with degree $\pi^!\beta$. Since $\pi^!\beta\cdot E=0$, we may consider the admissible graph $\Psi^\circ$ with a single vertex of genus $g$, $k$ legs, no roots, and degree $\pi^!\beta$. Let
\begin{equation*}
   s: \overline\cM_{\Psi^\circ}(\wtil X, E) \to \overline\cM_{g,k}(\wtil X,\pi^!\beta)
\end{equation*}
be the proper stabilization map. The virtual dimension of $\overline\cM_{\Psi^\circ}(\wtil X, E)$ coincides with that of $\overline\cM(X)$ since there are no relative markings. Recall the schematic diagram of absolute--relative correspondences \eqref{eq:degn scheme}. Our goal in this section is to refine the first correspondence to the level of virtual cycles. The following naive virtual cycle identity need not hold in positive genus, even for a point blow-up (see Remark~\ref{rmk:counterex for g=1}):
\begin{equation*}
    \pi_*[\overline\cM_{g,k}(\wtil X,\pi^!\beta)]^{\vir}=[\overline\cM_{g,k}(X,\beta)]^{\vir},\qquad g>0.
\end{equation*}
 We now compare the virtual cycle of $\overline\cM_{\Psi^\circ}(\wtil X, E)$ with no relative markings and the virtual cycle of $\overline\cM(X)$.
 
 \begin{rmk}
     The terminology ``absolute--relative correspondence'' is inspired by the work of Tehrani and Zinger. In \cite[Theorem~1]{TZ16}, they prove, under suitable hypotheses including the $(g,\beta)$-hollowness of the relative divisor, that absolute Gromov--Witten invariants agree with the corresponding relative Gromov--Witten invariants. Although our relative divisor $E$ does not generally satisfy the $(g,\beta)$-hollowness condition, we will show below that an analogous equality nevertheless holds in our setting.
 \end{rmk}

We first extend the definition of positivity to genus one for all regular maps $f:C\to Z$ with $g(C)\leq 1$.
\begin{defn}
\label{defn:positivity for g=1}
We say that the blow-up center $Z \subset X$ is positive if it satisfies $c_1(N_{Z/X}) \cdot f_*[C] \geq0$ for any regular map $f : C\to Z$ with $g(C)\leq 1$.
\end{defn}
The following definition will be used to refine our result later.

\begin{defn}
\label{defn:supportive}
    We say that the blow-up center $Z\subset X$ is supportive with respect to $\beta$ if it is positive and $\beta$ is in the image of $\iota_*:\rH_2(Z)\to \rH_2(X)$, that is, there is an effective class $\beta_Z$ such that $\iota_*\beta_Z=\beta$ and
    \begin{equation*}
        c_1(N_{Z/X}) \cdot \beta_Z = 0.
    \end{equation*}
    In this case, define
    \begin{equation*}
        B_Z(\beta)=\{\text{effective }\beta_Z\in\rH_2(Z;\Z)\mid \iota_*\beta_Z=\beta\text{ and }c_1(N_{Z/X})\cdot \beta_Z=0\}
    \end{equation*}
    as the set of effective classes in $Z$ that are ``supportive'' with respect to $\beta$.
\end{defn}

With the notation fixed above, we now compare the relative virtual cycle with $[\overline\cM(X)]^{\vir}$ using an equivariant version of Li's degeneration formula. Positivity is understood in the sense of Definition~\ref{defn:positivity} in genus zero and Definition~\ref{defn:positivity for g=1} in genus one; the set $B_Z(\beta)$ is as in Definition~\ref{defn:supportive}

Our main ingredient is the degeneration formula of Li. Let $W\to B$ be a flat degeneration over a smooth curve, let $\fW\to B$ be its stack of expanded degenerations, and denote the fiber of $\fW$ over $0\in B$ by $\fW_0$. Fix a discrete data $\mathbf d=(g,k,b)$.

\begin{thm}[{{\cite[Theorem 3.15]{Li02}}}]
\label{thm: degeneration formula}
    Assume that the singular fiber of $W\to B$ is $W_0=Y_1\cup_DY_2$.
    Then, we have
    \begin{equation*}
        [\overline\cM(\fW_0,\mathbf d)]^{\vir}=
        \sum_{\eta\in\overline\Omega}
        \frac{\mathbf m(\eta)}{|\Eq(\eta)|}
        {\Phi_\eta}_*\Delta^!\left([\overline\cM_{\Psi_1}(Y_1,D)]^{\vir}
        \times [\overline\cM_{\Psi_2}(Y_2,D)]^{\vir}\right),
    \end{equation*}
    where $\mathbf m(\eta)$ and $|\Eq(\eta)|$ are constants associated to $\eta$. Following Li's notation, $\overline\Omega$ is the set of equivalence classes of admissible triples $(\Psi_1,\Psi_2,I)$, $q$ is the common number of roots of $\Psi_1$ and $\Psi_2$, and $\Phi_\eta$ is the gluing morphism for relative stable maps. The Gysin pullback along the diagonal morphism $
        \Delta:D^q\longrightarrow D^q\times D^q$
    is taken with respect to the relative evaluation maps $        \overline\cM_{\Psi_1}(Y_1,D)\to D^q$ and $       \overline\cM_{\Psi_2}(Y_2,D)\to D^q$.
\end{thm}

Now recall the space $W=\Bl_{Z\times\{0\}}(X\times\PP^1)$ whose fiber over $\infty$ is naturally isomorphic to $X$, while its central fiber is
\begin{equation*}
    W_0=\wtil X\cup_E\what E,\qquad
    \what E\coloneqq\PP_Z(N_{Z/X}\oplus1).
\end{equation*}
Let $\fW\to\PP^1$ be the associated stack of expanded degenerations. We abbreviate the moduli of pre-deformable stable maps to its fibers by
\begin{equation*}
    \overline\cM(\fW/\PP^1)
    \coloneqq
    \overline\cM(\fW,(g,k,\pi^!\beta)).
\end{equation*}
The blow-down morphism $W\to X\times\PP^1$, followed by projection to $X$ and stabilization, induces a proper morphism
\begin{equation*}
    \wtil p:\overline\cM(\fW/\PP^1)\longrightarrow\overline\cM(X).
\end{equation*}

Recall that $0:= [1:0] \in \PP^1$ and $\infty:= [0:1] \in \PP^1$. Consider a $\C_u^*$-action on $\PP^1$ given by $\lambda \cdot [a:b] := [\lambda a : b]$ for $\lambda \in \C^*$ and $[a:b] \in \PP^1$. This action induces a $\C_u^*$-action on the degeneration and on the moduli of maps to its expanded fibers. It acts trivially on $\wtil X$ and induces the fiberwise action on $\what E=\PP_Z(N_{Z/X}\oplus1)$ fixing $E=\PP_Z(N_{Z/X}\oplus0)$ and $\PP_Z(0\oplus1)\cong Z$ pointwise. We use the subscript $u$ for $\C_u^*$ to distinguish it from the earlier action on the moduli space of ordinary stable maps.

The proof of Theorem~\ref{thm: degeneration formula}, carried out equivariantly as in \cite[Section 4]{Li02}, gives the equivariant degeneration formula
\begin{multline}
\label{eq:equivariant degeneration formula}
    [\overline\cM(\fW_0,(g,k,\pi^!\beta))]^{\vir,\C_u^*}=\\
    \sum_{\eta\in\overline\Omega} \frac{\mathbf m(\eta)}{|\Eq(\eta)|}
    {\Phi_\eta}_* \Delta^!\left([\overline\cM_{\Psi_1}(\wtil X,E) \times \overline\cM_{\Psi_2}(\what E,E)]^{\vir,\C_u^*}\right),
\end{multline}
where 
$$
\Delta^! : A_*( \overline\cM_{\Psi_1}(\wtil X,E) \times \overline\cM_{\Psi_2}(\what E,E))^{\C_u^*} \to A_*(\overline\cM_{\Psi_1}(\wtil X,E) \times_{E^q} \overline\cM_{\Psi_2}(\what E,E))^{\C_u^*}
$$
is the equivariant Gysin pull-back. Recall that $0$ and $\infty$ are the fixed points of the $\C_u^*$-action. We have $[\infty]^{\C_u^*} = h + u$ and $[0]^{\C_u^*} = h$ in $\rH^*(\PP^1)^{\C_u^*}$, where $h$ is the hyperplane class. Let
\begin{gather}
    \jmath_\infty:
    \overline\cM(X)\longrightarrow\overline\cM(\fW/\PP^1),
    \qquad \jmath_0: \overline\cM(\fW_0,(g,k,\pi^!\beta))
    \longrightarrow\overline\cM(\fW/\PP^1)
\end{gather}
be the inclusions of the fibers. 
Compatibility of the family obstruction theory with the fiber virtual classes gives
\begin{align*}
(\jmath_\infty)^{\C_u^*}_*[\overline\cM(X)]^{\vir,\C_u^*} &= [\overline\cM(\fW/\PP^1)]^{\vir,\C_u^*} \cap [\infty]^{\C_u^*},  \\
(\jmath_0)^{\C_u^*}_*[ \overline\cM(\fW_0,(g,k,\pi^!\beta))]^{\vir,\C_u^*} &= [\overline\cM(\fW/\PP^1)]^{\vir,\C_u^*} \cap [0]^{\C_u^*}.
\end{align*}
Using $[\infty]^{\C_u^*} - [0]^{\C_u^*} = u$ and \eqref{eq:equivariant degeneration formula}, we obtain
\begin{equation}
\label{eq:equivariant fiber relation}
\begin{split}
    &{\jmath_\infty}_*
    [\overline\cM(X)]^{\vir,\C_u^*}\\
    &\quad-
    \sum_{\eta\in\overline\Omega}\frac{\mathbf m(\eta)}{|\Eq(\eta)|}{(\jmath_0\circ\Phi_\eta)}_*
    \Delta^!
    \left([\overline\cM_{\Psi_1}(\wtil X,E) \times \overline\cM_{\Psi_2}(\what E,E)]^{\vir,\C_u^*}
    \right)\\
    &=
    u\,[\overline\cM(\fW/\PP^1)]^{\vir,\C_u^*}.
\end{split}
\end{equation}
Here it is important that $\overline\cM(\fW/\PP^1)$ parametrizes maps to expanded fibers and is endowed with its family perfect obstruction theory. Thus, the right-hand side of \eqref{eq:equivariant fiber relation} contains a single copy of the base weight $u$.\footnote{If one instead uses ordinary stable maps to the total space $W$, the moving normal theory contains $R\rho_*\cO_C\otimes T_{\PP^1}$, and this argument does not apply.}

We now compute the virtual dimension of a $\C_u^*$-fixed locus of $\overline\cM_{\Psi_2}(\what E,E)$. For a geometric point $f:C\to\what E[l]$ in the simple fixed locus, the source curve can be decomposed as follows:
\begin{equation*}
    C=C_Z\cup C_{\edg},
\end{equation*}
where $f$ restricts to $f_Z:C_Z\to Z$ and $f_{\edg}$ maps $C_{\edg}$ to torus-invariant lines in $\what E$ connecting $E$ and $\PP_Z(0\oplus1)\cong Z$. Note that $C_Z$ and $C_{\edg}$ are not necessarily connected. Let $h$ be the number of irreducible components of $C_{\edg}$, and let $\mathfrak e$ ($\geq h$) be the total degree of $C_{\edg}$. The total degree associated with $C_Z$ will be denoted by $\beta_*\in\rH_2(Z;\Z)$. Let $m$ be the total number of $*$-vertices (including unstable ones), $g_*$ be the total genus of the connected components of $C_Z$, and $k_*$ be the total number of markings assigned to $*$-vertices. Let $\Gamma$ be the graph associated to $f$, and $\overline\cM_\Gamma$ be the fixed locus containing $f$. Let $U_i$ be the number of univalent ($i=1$) and bivalent $(i=2)$ unstable $*$-vertices. Relative roots on $E$ are not counted as unstable vertices. Let
\begin{equation*}
    a_\Gamma^{\mathrm{Simp}}=\vcd\overline\cM_\Gamma\quad\text{and}\quad\delta_\Gamma=c_1(N_{Z/X})\cdot\beta_*.
\end{equation*}
Recall the definitions of simple and composite loci from \cite{GV05}. A composite fixed map can be decomposed into simple components and composite rubber parts.

\begin{lem}
\label{lem:vdim for degn}
The virtual codimension of a simple fixed locus $\overline\cM_\Gamma$ in $\overline\cM_{\Psi_2}(\what E,E)$ is given by
\begin{equation}
\label{eq:vdim for direct degn}
    a_\Gamma^{\mathrm{Simp}}=\delta_\Gamma+r(m-g_*)
    +h+r(\mathfrak e-h)-2U_1-U_2.
\end{equation}
Equivalently, if $S_*=m-U_1-U_2$ is the number of stable $*$-vertices, then
\begin{equation}
\label{eq:rearranged vdim for direct degn}
    a_\Gamma^{\mathrm{Simp}}=\delta_\Gamma+r(S_*-g_*)+(r-2)U_1+(r-1)U_2+h+r(\mathfrak e-h).
\end{equation}
For composite fixed loci, the full virtual codimension is $a_{\Gamma'}^{\mathrm{Simp}}+1$, where $\Gamma'$ is associated to its simple part. The additional $1$ corresponds to a target-smoothing factor (see \cite[Section 3.3]{GV05}).
\end{lem}

\begin{proof}
Let $\mathbb E$ be the relative perfect obstruction theory of $\overline\cM_{\Psi_2}(\what E,E)$. We compute the virtual dimension of a simple fixed locus in $\overline\cM_{\Psi_2}(\what E,E)^{\C_u^*}$. We temporarily assume that all vertices are stable. Let $\fM_{g,k}'$ be the moduli of genus $g$ prestable curves (which are possibly disconnected) with $k$ markings. Its prescribed number of connected components will be clear from the context. The contribution of $C_Z$ is
\begin{equation*}
    \vd\fM_{g_*,k_*+h}'+\rk\mathbb E\vert_{C_Z}^{\fix}=(m-g_*)(d-r-3)+h+\int_{\beta_*}c_1(Z)+k_*.
\end{equation*}
The contribution of $C_{\edg}$ to the virtual dimension is $h(d-1)$. Each node connecting $C_Z$ and $C_{\edg}$ imposes a matching condition of codimension $\dim Z=d-r$. The node-smoothing parameters have nonzero $\C_u^*$-weights. Therefore,
\begin{equation*}
    \vd\overline\cM_\Gamma=(m-g_*)(d-r-3)+\int_{\beta_*}c_1(X)
    -\delta_\Gamma+hr+k_*+2U_1+U_2.
\end{equation*}
On the other hand, denoting the corresponding simple graph by $\Psi_2^{\mathrm{Simp}}$, the virtual dimension of the corresponding relative moduli space is
\begin{equation*}
    \vd\overline\cM_{\Psi_2^{\mathrm{Simp}}}(\what E,E)=(m-g_*)(d-3)+\int_{\beta_*}c_1(X)+r\mathfrak e+h+k_*.
\end{equation*}
Comparing the two virtual dimensions proves \eqref{eq:vdim for direct degn}. Substituting $m=S_*+U_1+U_2$ gives \eqref{eq:rearranged vdim for direct degn}. The statement for composite loci follows from the additional target smoothing factor.
\end{proof}

After applying $\wtil p_*$, we divide \eqref{eq:equivariant fiber relation} by $u$ and take the $u^{-1}$-coefficient. Since $\C_u^*$ acts trivially on $\overline\cM(X)$, we have $A_*^{\C_u^*}(\overline\cM(X))=A_*(\overline\cM(X))[u]$. In particular, the class $\wtil p_*[\overline\cM(\fW/\PP^1)]^{\vir,\C_u^*}$ has zero $u^{-1}$-coefficient. Consequently,
\begin{multline}
\label{eq:residue of equivariant degeneration}
    [\overline\cM(X)]^{\vir}=\\
    [u^0]\sum_{\eta\in\overline\Omega}
    \frac{\mathbf m(\eta)}{|\Eq(\eta)|}
    {(\wtil p\circ\jmath_0\circ\Phi_\eta)}_*
    \Delta^!\left([\overline\cM_{\Psi_1}(\wtil X,E) \times \overline\cM_{\Psi_2}(\what E,E)]^{\vir,\C_u^*}
    \right).
\end{multline}
Here and below, $[u^0]$ is the $u^0$-coefficient taken after expanding the inverse equivariant Euler classes in Laurent series in $u^{-1}$. Using \eqref{eq:residue of equivariant degeneration}, we obtain absolute--relative correspondence in genus at most one.

\begin{thm}
\label{thm: main Abs-Rel corr}
Assume $g\leq1$ and that $Z$ is a positive blow-up center, in the sense of Definition~\ref{defn:positivity} when $g=0$ and Definition~\ref{defn:positivity for g=1} when $g=1$. There are three cases.
\begin{itemize}
    \item [(1)] The genus is zero, $g=0$;
    \item [(2)] The genus is one, $g=1$, and $Z$ is not supportive with respect to $\beta$;
    \item [(3)] The genus is one, $g=1$, and $Z$ is supportive with respect to $\beta$.
\end{itemize}
In cases (1) and (2), we have
    \begin{equation*}
        (\pi\circ s)_*[\overline\cM_{\Psi^\circ}(\wtil X,E)]^{\vir}=[\overline\cM(X)]^{\vir}.
    \end{equation*}
In case (3), we have
 \begin{equation*}
        (\pi\circ s)_*[\overline\cM_{\Psi^\circ}(\wtil X,E)]^{\vir}=[\overline\cM(X)]^{\vir}-\sum_{\beta_Z\in B_Z(\beta)}
        \iota_*[\overline\cM_{1,k}(Z,\beta_Z)]^{\vir}.
    \end{equation*}
These are equalities in $A_*(\overline\cM(X))$, where $\iota:Z\xhookrightarrow{}X$ is the inclusion.
\end{thm}

\begin{proof}
We first claim that every term in \eqref{eq:residue of equivariant degeneration} with a nonempty $\what E$-side not entirely contained in $\PP_Z(0\oplus1)\cong Z$ vanishes. Equip each summand
\begin{equation*}
    \overline\cM_{\Psi_1}(\wtil X,E)
    \times_{E^q}
    \overline\cM_{\Psi_2}(\what E,E)
\end{equation*}
with the action that is trivial on $\overline\cM_{\Psi_1}(\wtil X,E)$. Therefore, its fixed loci correspond to the fixed loci of $\overline\cM_{\Psi_2}(\what E,E)$.

Since $g\leq1$, we have $S_*-g_*\geq0$. With the positivity assumption, this gives $\delta_\Gamma\geq0$: decompose $C_Z$ into its nonconstant irreducible components and apply positivity to their normalizations, while the contracted components contribute zero. 
For a composite fixed locus $M_\Gamma$, let $\Gamma'$ be the graph corresponding to its simple part. Note that $\Gamma'$ can be empty when the image of the map lies entirely on the rubber part. In this case, we may assume that $a_\emptyset^{\mathrm{Simp}} = 0$, so we have $a_{\Gamma} = a_{\emptyset}^{\mathrm{Simp}} + 1  = 1 > 0.$
Now we assume that the simple part $\Gamma'$ is not empty.
Moreover, if the image is not contained in $\PP_Z(0\oplus1)\cong Z$, then $h>0$. Hence, equation \eqref{eq:rearranged vdim for direct degn} gives
\begin{equation*}
    a_\Gamma^{\mathrm{Simp}}\geq h>0.
\end{equation*}
None of these fixed loci contributes to the $u^0$-coefficient of \eqref{eq:residue of equivariant degeneration}. It remains to consider maps whose image is entirely contained in $\PP_Z(0\oplus1)\cong Z$. Write $\beta_Z=\beta_*$ for the degree. In this case, Lemma~\ref{lem:vdim for degn} gives
\begin{equation*}
    a_\Gamma^{\mathrm{Simp}}=\delta_\Gamma+r(1-g).
\end{equation*}
If $g=0$, then $a_\Gamma^{\mathrm{Simp}}=\delta_\Gamma+r>0$, so this contribution vanishes. If $g=1$, then it vanishes unless $ \delta_\Gamma=c_1(N_{Z/X})\cdot\beta_Z=0$. Therefore, these fixed loci contribute exactly
\begin{equation*}
    \iota_*[\overline\cM_{1,k}(Z,\beta_Z)]^{\vir}
\end{equation*}
for $\beta_Z\in B_Z(\beta)$. Since the summand with empty $\what E$-side is canonically identified with $[\overline\cM_{\Psi^\circ}(\wtil X,E)]^{\vir}$, substituting the surviving terms into \eqref{eq:residue of equivariant degeneration} proves the theorem.
\end{proof}

\begin{rmk}
\label{rmk:counterex for cycle}
    The correction terms of Theorem~\ref{thm: main Abs-Rel corr} are necessary. Recall $X$, $Z$, and $\beta$ from Remark~\ref{rmk:counterex for case ii}. Set $g=1$, and let $E'$ be the exceptional divisor of $\wtil S$. Then, $\pi_*c_2(T_{\wtil S}(-\log E'))=c_2(T_S)-[\pt]$ yields
    \begin{equation*}
        (\pi\circ s)_*[\overline\cM_{\Psi^\circ}(\wtil X, E)]^{\vir}=[\overline\cM(X)]^{\vir}-[\pt]
    \end{equation*}
    in $A_0(S)$, under the identification $\overline\cM_{1,0}(X,\beta)\cong S$.
    This shows that the naive virtual cycle identity of the absolute--relative correspondence fails in this example. It is compatible with the theorem since $Z$ is supportive.
\end{rmk}

We next extend the definition of positivity to any genus $g$.
\begin{defn}
\label{defn:positivity for all g}
We say that the blow-up center $Z \subset X$ is $g$-positive if it satisfies 
\begin{equation}
\label{eq:higher genus positivity}
    c_1(N_{Z/X})\cdot f_*[C]>(g-1)r
\end{equation}
for any regular map $f : C\to Z$ with $g(C)\leq g$ and $f_*[C]\neq 0$.
\end{defn}
Under this positivity condition, the higer-genus statement is as follows.

\begin{thm}
\label{thm: all genus Abs-Rel corr}
Fix a genus $g\geq2$. Assume $\dim Z\geq4$ and that $Z$ is a $g$-positive blow-up center. Then,
\begin{equation*}
    (\pi\circ s)_*[\overline\cM_{\Psi^\circ}(\wtil X,E)]^{\vir}=[\overline\cM_{g,k}(X,\beta)]^{\vir}
\end{equation*}
in $A_*(\overline\cM_{g,k}(X,\beta))$.
\end{thm}

\begin{proof}
We again show that the contribution of every fixed locus with a nonempty $\what E$-side vanishes. First, assume $\beta_*\neq0$. Then assumption \eqref{eq:higher genus positivity} gives
\begin{equation*}
    \delta_\Gamma+r(S_*-g_*)>0.
\end{equation*}
Since $r\geq2$, equation \eqref{eq:rearranged vdim for direct degn}  shows $a_\Gamma^{\mathrm{Simp}}>0$, including the locus whose image is entirely contained in $Z$. Hence, all nonzero-degree $*$-loci have zero $u^0$-contribution.

Now suppose $\beta_*=0$. Since $\beta_*$ is a sum of effective vertex degrees, every $*$-vertex has degree zero. If every stable $*$-vertex has genus zero or one, then $S_*-g_*\geq0$. Since any such simple locus with nonempty $\what E$-side has $h>0$, \eqref{eq:rearranged vdim for direct degn} again gives $a_\Gamma^{\mathrm{Simp}}\geq h>0$. This shows the required vanishing for such correction terms.

Otherwise, there is a degree-zero stable $*$-vertex of genus $s\geq2$. Its vertex factor has virtual cycle
\begin{equation*}
    [\overline\cM_{s,\mathbf k}(Z,0)]^{\vir}=c_{s\cdot \dim Z}
    \left(\mathbb E_{\mathrm{Hdg},\mathbf k}^\vee\boxtimes T_Z
    \right)\cap [\overline\cM_{s,\mathbf k}\times Z]
\end{equation*}
for some non-negative integer $\mathbf k$. The subscript of the Hodge bundle indicates the curve moduli base. Let $q_\mathbf k:\overline\cM_{s,\mathbf k}\times Z\to\overline\cM_s\times Z$ be the morphism forgetting all markings and flags. By compatibility of the Hodge bundle with stabilization, $\mathbb E_{\mathrm{Hdg},\mathbf k}=q_\mathbf k^*\mathbb E_{\mathrm{Hdg},0}$. Pullbacks to products with $Z$ are suppressed. Consequently, 
\begin{equation*}
    [\overline\cM_{s,n}(Z,0)]^{\vir}=q_n^*
    [\overline\cM_{s,0}(Z,0)]^{\vir}.
\end{equation*}
Now  $\vd\overline\cM_{s,0}(Z,0)=(1-s)(\dim Z-3)<0$ shows
\begin{equation*}
    [\overline\cM_{s,n}(Z,0)]^{\vir}=0
\end{equation*}
for every $n$. Since a zero vertex class forces the virtual class of a fixed graph to vanish, the only contribution surviving in \eqref{eq:residue of equivariant degeneration} is the main term with empty $\what E$-side. This proves the theorem.
\end{proof}

\begin{rmk}
The strict inequality in \eqref{eq:higher genus positivity} is needed in general for 
\begin{equation*}
    (\pi\circ s)_*[\overline\cM_{\Psi^\circ}(\wtil X,E)]^{\vir}=[\overline\cM_{g,k}(X,\beta)]^{\vir}
\end{equation*}
to hold. If the strict inequality in \eqref{eq:higher genus positivity} is replaced by a weak inequality, the locus associated with a nonzero class $\beta_Z$ satisfying
\begin{equation*}
    \iota_*\beta_Z=\beta,\qquad
    c_1(N_{Z/X})\cdot\beta_Z=(g-1)r
\end{equation*}
has virtual codimension zero and contributes $\iota_*[\overline\cM_{g,k}(Z,\beta_Z)]^{\vir}$. Analogously, define the set of supportive effective classes
\begin{equation*}
        B_Z^g(\beta)=\{\text{effective }\beta_Z\in\rH_2(Z;\Z)\mid \iota_*\beta_Z=\beta\text{ and }c_1(N_{Z/X})\cdot \beta_Z=(g-1)r\}
    \end{equation*}
as in Definition~\ref{defn:supportive}. Then, the second statement of Theorem~\ref{thm: main Abs-Rel corr} carries over as follows:
 \begin{equation*}
        (\pi\circ s)_*[\overline\cM_{\Psi^\circ}(\wtil X,E)]^{\vir}=[\overline\cM(X)]^{\vir}-\sum_{\beta_Z\in B_Z^g(\beta)}
        \iota_*[\overline\cM_{g,k}(Z,\beta_Z)]^{\vir},
 \end{equation*}
assuming $\dim Z\geq 4$ and $c_1(N_{Z/X})\cdot f_*[C]\geq(g-1)r$ for any regular map $f : C\to Z$ with $g(C)\leq g$ and $f_*[C]\neq 0$.
\end{rmk}

Together, Theorems~\ref{thm: main Abs-Rel corr} and \ref{thm: all genus Abs-Rel corr} give the stated absolute--relative correspondence in every genus. A counterexample to the corresponding statement without the dimension assumption when $g\geq2$ will be given in the next section.

\subsection{Counterexample to Conjecture \ref{conj:Hu}}
\label{sec:counterex of Hu's conj}
We apply the degeneration formula of Li (Theorem~\ref{thm: degeneration formula}). We give a counterexample to Hu's point-blow-up conjecture \cite{Hu00} in genus $g\geq 2$. Here we provide a counterexample to this conjecture.
    Set $g=2$ and  $X=\PP^5$, and let $\beta$ be the line class of $X$. Let $\wtil X$ denote the point blow-up of $X$, with exceptional divisor $E\cong \PP^4$. Consider stable maps with a single absolute marking. Let $\Psi_1$ be an admissible graph for relative stable maps to $(\wtil X, E)$ with $n_{\wtil X}$ vertices with total genus $g_{\wtil X}$, a single leg, and $q$ roots. Let $\Psi_2$ be an admissible graph for relative stable maps to $(\what E, E)$ with no legs, where $n_{\what E}, g_{\what E}$, and $q$ are defined analogously.  We will use the notation $$\langle \gamma_1,\cdots,\gamma_k\mid \zeta_1,\cdots, \zeta_q\rangle^{Y, D}_{\Psi}$$ to denote the relative GW invariant of $(Y,D)$ with absolute insertions $\gamma_i$ and relative insertions $\zeta_j$. Let $\boldsymbol{\mu}$ be the sum of multiplicities at the $q$ roots.
    Then, a direct computation shows
\begin{equation*}
    \vd M_{\Psi_1}(\wtil X,E)=2(n_{\wtil X}-g_{\wtil X})+q -5\boldsymbol{\mu} +7,\quad\vd M_{\Psi_2}(\what E, E)=2(n_{\what E}-g_{\what E})+q + 5\boldsymbol{\mu},
\end{equation*}
and $\vd M(X)= 5$. Consider the degeneration $X\rightsquigarrow\wtil X\cup_E \what E$. Since $\pi^!\beta-2\cdot \ell$ is not effective, there are only two possible contributions in the degeneration formula for $\langle \pt\rangle^{X}_{2,\beta}$. The first main contribution is $\langle \pt\mid\text{ }\rangle^{\wtil X, E}_{\Psi^\circ}$. The second arises from $n_{\wtil X}=n_{\what E}=q=1$. For dimensional reasons, we get $g_{\wtil X}=0$ and $g_{\what E}=2$. Hence, the correction term is 
\begin{equation*}
    \langle \pt\mid 1\rangle^{\wtil X, E}_{\Psi_1}\cdot \langle \text{ }\mid\pt\rangle^{\what E, E}_{\Psi_2}.
\end{equation*}
The degree of $\Psi_1$ is the class of the proper transform of a line in $X$, and the degree of $\Psi_2$ is $\ell$. Hence, $\boldsymbol{\mu} = \ell \cap E = 1$. Relative localization (see \cite{GV05}) gives
\begin{equation*}
\langle \pt\mid 1\rangle^{\wtil X, E}_{\Psi_1}=1\quad\text{and}\quad
    \langle \text{ }\mid\pt\rangle^{\what E, E}_{\Psi_2}=\int_{\overline\cM_{2,1}}\frac{\{\Lambda_2^\vee(-1)\}^5}{1+\psi_1}=\frac{1}{144}
\end{equation*}
using Hodge integrals \cite{FP00}. Here $\Lambda_g^\vee(t)$ denotes $t^g-\lambda_1t^{g-1}+\cdots +(-1)^g\lambda_g$. Next, consider the degeneration $\wtil X\rightsquigarrow\wtil X\cup_E \PP_E(\cO_E(-1)\oplus 1)$. Again, for dimensional reasons, there is a single correction term 
\begin{equation*}
    \langle \pt\mid 1\rangle^{\wtil X, E}_{\Psi_1}\cdot \langle \text{ }\mid\pt\rangle^{\PP_E(\cO_E(-1)\oplus 1), E}_{\Psi_2}.
\end{equation*}
The degree of $\Psi_1$ is the class of the proper transform of a line in $X$, and the degree of $\Psi_2$ is the line class of $\PP_E(\cO_E(-1)\oplus 0)\cong \PP^4$.
By localization, one computes
\begin{equation}
\label{eq:second correction}
    \langle \text{ }\mid\pt\rangle^{\PP_E(\cO_E(-1)\oplus 1), E}_{\Psi_2}=\frac{1}{128}.
\end{equation}
The calculation is given in Appendix B. Combining the two degenerations shows
\begin{equation*}
    \langle \pt\rangle^{\wtil X}_{2,\pi^!\beta}=\langle \pt\rangle^{X}_{2,\beta}+\frac{1}{1152}.
\end{equation*}
This disproves Conjecture~\ref{conj:Hu}.

\appendix
\section{Numerical blow-up formula in special cases} \label{sec:Appblowup1}
In this appendix, we give an algebro-geometric proof of cases (I), (II), and (IV) of Theorem~\ref{thm:Hublowup} \cite{Hu00, Hu01}. The proof replaces symplectic cutting techniques with degeneration formulas and Fredholm-index calculations with virtual dimension arguments. We retain the notations and conventions from Section~\ref{sec:higher genus} and impose the following assumptions.
\begin{itemize}
    \item [(i)] We assume one of the following conditions: 
        \begin{itemize}
            \item[(a)] $d=\dim(X) \le 3$
            \item[(b)] $g=2$, $d = 4$, and $\dim(Z)=0$
            \item[(c)] $g=1$, $d \ge 4$, and $\dim(Z) \le 1$
            \item[(d)] $g=0$, $d \ge 4$, and $\dim(Z) \le 2$
        \end{itemize}
        \medskip
        For cases (a) and (c), we further assume $k\ge 1$ or there is no regular map from a curve to $Z$ with degree $\beta$.
    \item [(ii)] Assume $\gamma_1|_Z = \cdots = \gamma_k|_Z = 0$. In other words, all insertion classes $\gamma_i$ restrict trivially to the blow-up locus $Z$.
    \item [(iii)] $c_1(X) \cdot f_*[C] \geq0$ for any regular map $f : C\to Z$ from a curve $C$.
\end{itemize}
For example, the third assumption holds if $Z$ is a point or if $X$ is Fano or Calabi-Yau. The following theorem covers the algebro-geometric analogues considered in Theorem~\ref{thm:Hublowup}. 
 \begin{thm}
 \label{thm: numerical blowup}
     Under the above assumptions, the numerical blow-up formula holds:
     \begin{equation*}
         \langle\pi^*\gamma_1,\cdots,\pi^*\gamma_k\rangle_{g,\pi^!\beta}^{\wtil X}=\langle\gamma_1,\cdots,\gamma_k\rangle_{g,\beta}^{X}
     \end{equation*}
     for $\gamma_i\in \rH^*(X; \Z)$.  
 \end{thm}
 \begin{proof}
     We show that all correction terms vanish in each of the two degenerations in \eqref{eq:degn scheme}. We may assume the degrees of $\gamma_i$ satisfy
     \begin{equation*}
        \vd \overline\cM_{g,k}(X,\beta)=(1-g)(d-3)+\int_\beta c_1(X)+k=\sum \deg(\gamma_i),
     \end{equation*}
     since otherwise both sides vanish for dimensional reasons and assumption~(ii).
     Use the notation $n_{\wtil X}$, $g_{\wtil X}$, $n_{\what E}$, $g_{\what E}$ and $q$ from Section~\ref{sec:counterex of Hu's conj}. Let $\mathfrak e=\sum_{\mathrm{roots}}\mu$ and $\beta_{\wtil X}=\pi^!\beta -\mathfrak e\cdot\ell - \pi^! \iota_*\beta_*$ be the total degree of $\wtil X$-vertices, where $\beta_* \in H_2(Z;\Z)$ is an effective curve class of $Z$. Since $\gamma_1|_Z = \cdots = \gamma_k|_Z = 0$, only terms with all absolute markings on the $\Psi_1$-side can contribute. A direct computation shows
     \begin{equation*}
         \vd\overline \cM_{\Psi_1}(\wtil X, E)=(n_{\wtil X}-g_{\wtil X})(d-3)+\int_{\beta - \beta_*} c_1(X)+k+q-r\mathfrak e.
     \end{equation*}
     The relative insertions must account for the following dimension difference:
     \begin{multline*}
         \vd \overline \cM_{\Psi_1}(\wtil X, E)-\vd \overline\cM_{g,k}(X,\beta)\\=(-n_{\what E}+g_{\what E}+q)(d-3)+q-r\mathfrak e - \int_{\beta_*} c_1(X).
     \end{multline*}
     A term can contribute only if this dimension difference is non-negative. By assumption (iii), we have $\int_{\beta_*} c_1(X) \ge 0$.
     Together with $\mathfrak e\geq q$, we get 
     \begin{equation}\label{eq: correctionineq}
     (-n_{\what E}+g_{\what E}+q)(d-3)\geq q(r-1).
     \end{equation}

     Now, by assumption (i), we have the following cases:
     \begin{itemize}
         \item [(a)] Assume $d \le 3$, and $k\ge 1$ or there is no regular map with degree $\beta$ from a curve to $Z$: 
         By assumption, there is no contribution with empty $\wtil X$-side. Hence, we may assume $q > 0$.
         Note that $(-n_{\what E}+g_{\what E}+q) \ge 0$. With $d\leq 3$, this implies $q=0$, which leads us to a contradiction.
         \item [(b)] Assume $g=2$, $\dim(X)=4$, $\dim(Z)=0$: 
         Note that the correction terms satisfy $n_{\what{E}} > 0$, $g_{\what{E}} \le 2$. Thus, we have $q+1 \ge 3q$, which implies $q=0$. Since $\dim Z=0$, this forces $k=0$ and $\beta=0$. One can thus see that the virtual dimension of $\overline\cM_{g,k}(X,\beta)$ is negative, a contradiction.
         \item [(c)] Assume $g = 1$, $\dim(Z) \le 1$, $d \ge 4$, and $k\ge 1$ or there is no regular map with degree $\beta$ from a curve to $Z$: 
         By assumption, maps with image contained in $Z \subset \what{E}$ make no contribution. Hence, we may assume $q > 0$.
         Moreover, the correction terms satisfy $n_{\what E}>0$. Since $g_{\what E}\leq 1$, we have $$q(d-3) \ge q(r-1) \ge q(d-2).$$ Since $d\geq 4$, this again implies $q=0$, which leads us to a contradiction.
         \item [(d)] Assume $g = 0$, $\dim(Z) \le 2$, and $d \ge 4$: 
         The correction terms satisfy $n_{\what E}>0$. With $g_{\what E} =  0$, we have $(q-1)(d-3) \ge q(d-3)$. This is impossible because $d \ge 4$.
     \end{itemize}
     
     In any case, there are no correction terms in the first degeneration. The argument carries over verbatim to the second degeneration, since the relative moduli term $\overline{\cM}_{\Psi_1}(\wtil X, E)$ is unchanged.
 \end{proof}

\begin{rmk}
By Theorem~\ref{thm: numerical blowup}, the blow-up formula \eqref{conj:blowup1} holds when $g=2$, $d = \dim(X) = 4$, and $Z$ is a point. This is why our first counterexample to Conjecture \ref{conj:Hu} has genus $2$ and a fivefold $X$.
\end{rmk}

In Theorem~\ref{thm:main2}, we proved $\pi_*[\overline\cM_{0,k}(\wtil X,\pi^!\beta+M\ell)]^{\vir}=0$ when the genus is $0$ and $Z$ is positive. The same method used in Theorem~\ref{thm: numerical blowup} yields a numerical vanishing statement in other genera under the same assumptions.
 \begin{thm}
 \label{thm: numerical vanishing}
      Under the assumptions of Theorem \ref{thm: numerical blowup} and for $M>0$, the following numerical vanishing holds:
     \begin{equation*}
\langle \pi^*\gamma_1,\cdots, \pi^*\gamma_k\rangle^{\wtil{X}}_{g,\pi^!\beta + M\ell}=0
\end{equation*}
for $\gamma_i\in \rH^*(X;\Z)$.
 \end{thm}
 \begin{proof}
Following the notation of Theorem~\ref{thm: numerical blowup}, we first have $\beta_{\wtil{X}} \cdot E \ge 0$ by pre-deformability. Together with $\beta_{\wtil{X}} = \pi^! \beta - \mathfrak{e} \ell - \pi^!\iota_*\beta_*$, we have $\mathfrak{e} = \beta_{\wtil{X}} \cdot E \ge 0$. Then the discrepancy is
\begin{multline*}
\vd \overline \cM_{\Psi_1}(\wtil X, E)-\vd \overline\cM_{g,k}(\wtil X,\pi^!\beta+M\ell)\\=(-n_{\what E}+g_{\what E}+q)(d-3)+q-r\mathfrak e -(r-1) M - \int_{\beta_*} c_1(X) \geq 0.    
\end{multline*}
Since $\mathfrak e\geq q$ and $\int_{\beta_*}c_1(X)\geq0$, the preceding inequality gives:
\begin{equation}
\label{eq: auxiliary}
(-n_{\what E}+g_{\what E}+q)(d-3) \ge (r-1)(q+M) \ge (r-1)q.
\end{equation}

Therefore, assuming (i) to (iii), the above inequality cannot hold except possibly when $\Psi_2 = \emptyset$, that is, $n_{\what E} = g_{\what E} = q = 0$. Even in this case, $n_{\what E} = g_{\what E} = q = 0$ cannot satisfy \eqref{eq: auxiliary} since $M> 0$. Thus all terms of the degeneration formula vanish numerically. Therefore we obtain the desired vanishing.     
 \end{proof}

\section{Details of the computation}
We give the computation details of \eqref{eq:second correction}. Endow $\PP(\cO_E(-1)\oplus 1)$ with its full torus action. Let $p_0,\cdots, p_4$ be the fixed points of $E \cong \PP^4$ and $\alpha_i$ be the negative of the weights of $\cO_{\PP^4}(-1)$ at $p_i$. Let $u$ be the weight of the fiber action on $\PP(\cO_E(-1)\oplus 1)$. Let $E_\infty := \PP(0 \oplus 1)$. Consequently, $u_i=u-\alpha_i$ are weights of $N_{E_\infty /\PP(\cO_E(-1)\oplus 1)}\cong \cO_{E_\infty}(-1)$ at $p_i$. Localization with respect to the fiber-scaling action gives the following simple-locus contribution to $\langle \text{ }\mid\pt\rangle^{(\PP(\cO_E(-1)\oplus 1), E_\infty)}_{\Psi_2}$:
\begin{align*}
    \int_{[\overline \cM_{2,1}(\PP^4,1)]^{\vir}}\frac{\ev_1^*[p_0]}{(u_0-\psi_1)\cdot e(R\rho_*f^*\cO_{E_{\infty}}(-1))}.
\end{align*}
Applying localization for the torus action on $\PP^4$, there are three types of graphs depending on the splitting of the genus $(g_0, g_j)$. Here $g_0$ and $g_j$ are the genera assigned to the vertices over $p_0$ and $p_j$, respectively. The non-constant part maps to the invariant line connecting $p_0$ and $p_j$. For convenience, we denote
\begin{equation*}
    \omega_{ij}=\alpha_i-\alpha_j,\quad E_{0j}=-\frac{1}{\omega_{0j}^2}\prod_{k\neq 0,j}\frac{1}{\omega_{0k}\,\omega_{jk}},\quad P_0=\prod_{k\neq 0}\omega_{0k}.
\end{equation*}
Here $\omega_{ij}$ denotes a tangent weight in $\PP^4$,  $E_{0j}$ the degree-one edge factor to $1/e(N^{\vir})$, and $P_0$ the equivariant point class $[\pt]$. There are three types of simple fixed loci, corresponding to the distribution of genus. We evaluate their contributions by substituting $\alpha_0=0, \alpha_1=\epsilon,\alpha_2=-\epsilon, \alpha_3=2\epsilon, \alpha_4=-2\epsilon$.\\

\textbf{Type~$(g_0,g_j)=(2,0)$.} There is a genus-two component over $p_0$. On each fixed locus, $e(-R\rho_*f^*\cO_{E_{\infty}}(-1))$ restricts to $\Lambda_2^\vee(u_0)$, and the sum of all contributions of this type is:
\begin{equation*}
\sum_{j=1}^4P_0E_{0j}\omega_{j0}\cdot\int_{\overline\cM_{2,2}}\frac{\prod_{k\neq 0}\Lambda_2^\vee(\omega_{0k})\cdot\Lambda_2^\vee(u_0)}{(\omega_{0j}-\psi_1)(u_0-\psi_2)}=\frac{7}{2560}.
\end{equation*}
Here $\psi_1$ is the cotangent class at the $p_0p_j$-edge flag. The $\psi_2$-class is the cotangent class at the ordinary marking.\\

\textbf{Type~$(g_0,g_j)=(1,1)$.} There is a genus-one component over each of $p_0$ and $p_j$. On each fixed locus, $e(-R\rho_*f^*\cO_{E_\infty}(-1))$ restricts to $\Lambda_1^\vee(u_0)\Lambda_1^\vee(u_j)$, and
the sum of all contributions of this type is:
\begin{equation*}
\sum_{j=1}^4P_0E_{0j}\cdot\int_{\overline\cM_{1,2}}\frac{\prod_{k\neq 0}\Lambda_1^\vee(\omega_{0k})\Lambda_1^\vee(u_0)}{(\omega_{0j}-\psi_1)(u_0-\psi_2)}\cdot \int_{\overline\cM_{1,1}}\frac{\prod_{k\neq j}\Lambda_1^\vee(\omega_{jk})\Lambda_1^\vee(u_j)}{\omega_{j0}-\psi_1}=\frac{11}{768}.
\end{equation*}
Again, the $\psi_1$-classes on $\overline\cM_{1,2}$ and $\overline\cM_{1,1}$ denote cotangent classes at the $p_0p_j$-edge flags.\\

\textbf{Type~$(g_0,g_j)=(0,2)$.} There is a genus-two component over $p_j$. On each fixed locus, $e(-R\rho_*f^*\cO_{E_{\infty}}(-1))$ restricts to $\Lambda_2^\vee(u_j)$, and
the sum of all contributions of this type is:
\begin{equation*}
\sum_{j=1}^4P_0E_{0j}\cdot\frac{1}{u_0+\omega_{0j}}\cdot\int_{\overline\cM_{2,1}}\frac{\prod_{k\neq j}\Lambda_2^\vee(\omega_{jk})\Lambda_2^\vee(u_j)}{\omega_{j0}-\psi_1}=\frac{236\epsilon^4+745\epsilon^2u^2-213u^4}{23040(\epsilon^2-u^2)(4\epsilon^2-u^2)}.
\end{equation*}
We use similar notation for the $\psi_1$- and $\psi_2$-classes as above. After setting $\epsilon =0$, or equivalently taking the $u\to \infty$, we get $-71/7680$.\\

The contributions of composite fixed loci in the relative localization formula \cite[Theorem 3.6]{GV05} contain an additional $\frac{1}{u+\psi_\sim}$ factor. One can check that every stable composite locus contributes $O(u^{-1})$. These contributions vanish in the limit $u\to \infty$. Thus, summing the three simple-locus contributions gives 
\begin{equation*}
    \frac{7}{2560}+\frac{11}{768}-\frac{71}{7680}=\frac{1}{128}.
\end{equation*}

\end{document}